\documentclass[a4paper,11pt]{amsart}

\usepackage[margin=1in]{geometry}
\usepackage{graphicx}
\usepackage{amsmath}
\usepackage[utf8]{inputenc}
\usepackage{amsfonts}
\usepackage{amsthm}
\usepackage{xcolor}
\usepackage[pdftex]{hyperref}

\usepackage{marginnote}

\numberwithin{equation}{section}

\newtheorem*{thm*}{Theorem}
\newtheorem{thm}{Theorem}[section]

\newtheorem{lemma}[thm]{Lemma}
\newtheorem{prop}[thm]{Proposition}

\newcommand{\R}{\mathbb{R}}

\newcommand{\C}{\mathbb{C}}
\newcommand{\N}{\mathbb{N}}

\newcommand{\X}{\mathcal{X}}
\newcommand{\HH}{\mathcal{H}}
\newcommand{\ZZ}{\mathcal{Z}}

\newcommand{\FF}{\mathcal{F}}
\newcommand{\NN}{\mathcal{N}}

\begin{document}
\nocite{*}
\title[Controllability of Coupled System of 2 KdV Equations]{Boundary Controllability of a nonlinear coupled system of two Korteweg--de Vries equations With Critical Size Restrictions on the spatial domain}
\author[Capistrano--Filho]{Roberto A. Capistrano--Filho}
\address{Department of Mathematics, Federal University of Pernambuco, UFPE, CEP 50740-545, Recife, PE, Brazil.}
\email{capistranofilho@dmat.ufpe.br}
\author[Gallego]{Fernando A. Gallego}
\address{Institute of Mathematics, Federal University of Rio de Janeiro, UFRJ, P.O. Box 68530, CEP 21945-970, Rio de Janeiro, RJ, Brazil.}
\email{fgallego@ufrj.br, ferangares@gmail.com}
\author[Pazoto]{Ademir F. Pazoto}
\address{Institute of Mathematics, Federal University of Rio de Janeiro, UFRJ, P.O. Box 68530, CEP 21945-970, Rio de Janeiro, RJ, Brazil.}
\email{ademir@im.ufrj.br}
\date{}

\begin{abstract}
This article is dedicated to improve the controllability results obtained by Cerpa \textit{et al.} in \cite{cerpapazoto2011} and by Micu \textit{et al.} in \cite{micuortegapazoto2009}
for a nonlinear coupled system of two Korteweg–de Vries (KdV) equations posed on a bounded interval. Initially, in \cite{micuortegapazoto2009}, the authors proved that the nonlinear system is exactly controllable by using four boundary controls without any restriction on the length $L$ of the interval. Later on, in \cite{cerpapazoto2011}, two boundary controls were considered to prove that the same system is exactly controllable for small values of the length $L$ and large time of control $T$. Here, we use the ideas contained in \cite{capisgallegopazoto2015} to prove that, with another configuration of four controls, it is possible to prove the existence of the so-called \textit{critical length phenomenon} for the nonlinear system, i. e., whether the system is controllable depends on the length of the spatial domain. In addition, when we consider only one control input, the boundary controllability still holds for suitable values of the length $L$ and time of control $T$. In both cases, the control spaces are sharp due a technical lemma which reveals a hidden regularity for the solution of the adjoint system.
\end{abstract}

\maketitle

\section{Introduction}

\subsection{Setting of the problem} In \cite{geargrimshaw1984}, a complex system of equations was derived by Gear and Grimshaw to model the strong interaction of
two-dimensional, long, internal gravity waves propagating on neighboring pycnoclines in a stratified fluid. It has the structure of a pair of Korteweg-de Vries equations
coupled through both dispersive and nonlinear effects and has been the object of intensive research in recent years. An interesting possibility now presents itself
is the study of the boundary controllability properties when the model is posed on a bounded domain.

In this paper, we are mainly concerned with the study of the the Gear-Grimshaw system

\begin{equation}
\label{gg1}
\begin{cases}
u_t + uu_x+u_{xxx} + a v_{xxx} + a_1vv_x+a_2 (uv)_x =0, & \text{in} \,\, (0,L)\times (0,T),\\
c v_t +rv_x +vv_x+abu_{xxx} +v_{xxx}+a_2buu_x+a_1b(uv)_x  =0,  & \text{in} \,\, (0,L)\times (0,T), \\
u(x,0)= u^0(x), \quad v(x,0)=  v^0(x), & \text{in} \,\, (0,L),
\end{cases}
\end{equation}
satisfying the following boundary conditions
\begin{equation}\label{gg2}
\begin{cases}
u(0,t)=h_0(t),\,\,u(L,t)=h_1(t),\,\,u_{x}(L,t)=h_2(t),\\
v(0,t)=g_0(t),\,\,v(L,t)=g_1(t),\,\,v_{x}(L,t)=g_2(t),
\end{cases}
\end{equation}
where $a_1, a_2, a, b, c, r\in \mathbb{R}$. We also assume that
$$1-a^2 b > 0 \quad \text{and} \quad  b, c > 0.$$
The functions $h_0, h_1, h_2, g_0, g_1$ and $g_2$ are the control inputs and $u_0, v_0$ the initial data.

The purpose is to see whether one can force the solutions of those systems to have
certain desired properties by choosing appropriate control inputs. Consideration will be given
to the following fundamental problem that arises in control theory:

\vglue 0.2 cm

\noindent\textbf{Exact Control Problem:} Given $T > 0$ and  $(u^0,v^0)$, $(u^1,v^1)$ in $(L^2(0,L))^2$, can one find appropriate $h_j$ and $g_j$, for $j=0,1,2$,
in a certain space such that the corresponding solution $(u,v)$ of \eqref{gg1}-\eqref{gg2} satisfies
\begin{equation}\label{exact}
u(x,T)=u^1(x) \quad \text{and} \quad v(x,T)=v^1(x)?
\end{equation}

\vglue 0.2 cm

If one can always find control inputs to guide the system from any given initial state $(u_0,v_0)$ to any given terminal state $(u_1,v_1)$, then the system
is said to be exactly controllable. However, being different from other systems, the length $L$ of the spatial domain may play a crucial role in determining the controllability of the system,
specially when some configurations of four controls input are allowed to be used. This phenomenon, the so-called \textit{critical length phenomenon}, was observed for the first time by Rosier \cite{rosier}
while studying the boundary controllability for the KdV equation. Throughout the paper we will provide a detailed explanation of such phenomenon but, roughly speaking,
Rosier proved the existence of a finite dimensional subspace $M$ of $L^2(0, L)$, which is not reachable by the KdV system, when starting from the origin, if $L$
belongs to a countable set of critical lengths.

\subsection{State of art}
As far as we know, the controllability results for system \eqref{gg1} was first obtained in \cite{micuortega2000}, when the model is posed on a periodic domain and $r = 0$. In this case, a diagonalization of the main terms allows to the decouple the corresponding linear system into two scalar KdV equations and use the previous results available in the literature. In what concerns a bounded interval $(0,L)$, later on, Micu \textit{et al.}, in  \cite{micuortegapazoto2009}, proved the following local exact boundary controllability property.
\vglue 0.2 cm
\noindent
{\bf Theorem A }(Micu \textit{et al.} \cite{micuortegapazoto2009}) {\em Let $L>0$ and $T >0$. Then, there exists a constant $\delta>0$, such that, for any initial and final data $(u^0,v^0), (u^1,v^1)\in (L^2(0, L))^2$ verifying $$||(u^0, v^0)||_{(L^2(0,L))^2}\leq\delta \quad\text{and}\quad ||(u^1, v^1)||_{(L^2(0,L))^2}\leq\delta,$$ there exist four control functions $h_1, g_1\in H^1_0(0, T)$ and $h_2, g_2 \in L^2(0,T)$, with $h_0=g_0=0$, such that the solution $$(u, v)\in C([0,T];(L^2(0,L))^2)\cap L^2(0,T;(H^1(0, L))^2)\cap H^1(0,T;(H^{-2}(0, L))^2)$$ of \eqref{gg1}-\eqref{gg2} verifies \eqref{exact}.}

\vglue 0.4 cm

The proof of Theorem A combines the analysis of the linearized system and the Banach’s fixed point theorem. It is important to point out that, in order to analyze the linearized system, the authors follows the classical duality approach \cite{dolecki,lions} and, therefore, the exact controllability property is equivalent to an observability inequality for the solutions of the adjoint system. The problem is then reduced to prove a nonstandard unique continuation property of the eigenfunctions of the corresponding differential operator.

An improvement of Theorem A was made by Cerpa \textit{et al.}, in \cite{cerpapazoto2011}. The authors considered the system \eqref{gg1}-\eqref{gg2} with only two control inputs acting on the Neumann boundary conditions, that is,
\begin{equation}\label{gg_CP}
\begin{cases}
u(0,t)=0,\,\,u(L,t)=0,\,\,u_{x}(L,t)=h_2(t),\\
v(0,t)=0,\,\,v(L,t)=0,\,\,v_{x}(L,t)=g_2(t).
\end{cases}
\end{equation}
In this case, the analysis of the linearized system is much more complicated, therefore the authors used a direct approach based on the multiplier technique that gives the observability inequality for small values of the length $L$ and large time of control $T$.
\vglue 0.2 cm
\noindent
{\bf Theorem B }(Cerpa \textit{et al.} \cite{cerpapazoto2011}) {\em Let us suppose that  $T,L>0$ satisfy $$1> \dfrac{\max\{b,c\}}{\min\left\lbrace b(1-\varepsilon^2), \left( 1 - \dfrac{a^2b}{\varepsilon^2} \right)\right\rbrace}\left\lbrace \frac{rL^2}{3c \pi^2}+ \frac{L^3}{3T\pi^2}\right\rbrace$$ where $$\varepsilon = \sqrt{ \dfrac{-(1-b)+ \sqrt{(1-b)^2+4a^2b^2}}{2b}}.$$ Then, there exists a constant $\delta>0$, such that, for any initial and final data $(u^0,v^0), (u^1, v^1)\in (L^2(0, L))^2$ verifying $$||(u^0, v^0)||_{(L^2(0,L))^2}\leq\delta \quad\text{and}\quad ||(u^1, v^1)||_{(L^2(0,L))^2}\leq\delta,$$ there exist two control functions $h_2, g_2 \in L^2(0,T)$, with  $h_0=g_0=h_1=g_1=0$, such that the solution $$(u, v)\in C([0,T];(L^2(0,L))^2)\cap L^2(0,T;(H^1(0, L))^2)\cap H^1(0,T;(H^{-2}(0, L))^2)$$ of \eqref{gg1}-\eqref{gg_CP} verifies \eqref{exact}.}
\vglue 0.4 cm
Although the analysis developed by the authors can be compared to the analysis developed by Rosier \cite{rosier} for the KdV equation, the problem related to the existence of critical lengths
addressed by Rosier was not studied, more precisely, the existence of the so-called \textit{critical length phenomenon}. Indeed, Rosier proved that the linear KdV equation is exactly controllable by means of a single boundary control except when $L$  lies in a countable set of critical lengths. This was done using the classical duality approach and the critical lengths found by rosier are such that there are eigenvalues of the linear problem for which the observability inequality leading to the controllability fails. More recently, the problem was investigated by Capistrano--Filho \textit{et al.}, in \cite{capisgallegopazoto2015}, considering a new set of boundary conditions, the Neumann boundary conditions
\begin{equation}\label{gg3}
\begin{cases}
u_{xx}(0,t)=h_0(t),\,\,u_x(L,t)=h_1(t),\,\,u_{xx}(L,t)=h_2(t),\\
v_{xx}(0,t)=g_0(t),\,\,v_x(L,t)=g_1(t),\,\,v_{xx}(L,t)=g_2(t),
\end{cases}
\end{equation}
getting the following result:
\vglue 0.2 cm
\noindent
{\bf Theorem C }(Capistrano--Filho \textit{et al.} \cite{capisgallegopazoto2015}) {\em Let $T>0$ and define the set
\begin{equation}
\FF_r:= \left\lbrace 2\pi k \sqrt{\frac{1-a^2b}{r}}: k \in \N^{*}\right\rbrace\cup \left\lbrace \pi \sqrt{\frac{(1-a^2b)\alpha(k,l,m,n,s)}{3r}}:k, l, m, n, s \in \N^{*} \right\rbrace,
\label{critical_f}
\end{equation}
where
\begin{align*}
\alpha:=\alpha(k,l,m,n,s)=&5k^2+8l^2+9m^2+8n^2+5s^2+8kl+6km\\
+&4kn+2ks+12ml+8ln+3ls+12mn+6ms+8ns.
\end{align*}
Consider the following positions of the control inputs and the boundary conditions \eqref{gg3}:
\begin{itemize}
\item[] $\vec{h}_1=(0,h_1,0),\,\, \vec{g}_1=(g_0,g_1,g_2)$\,\,  and \,\, $\vec{h}_2=(h_0,h_1,h_2),\,\, \vec{g}_2=(0,g_1,0)$,
\item[] $\vec{h}_3=(h_0,h_1,0),\,\,\vec{g}_3=(g_0,g_1,0)$\,\, and \,\, $\vec{h}_4=(0,h_1,h_2),\,\,\vec{g}_4=(0,g_1,g_2)$,
\item[] $\vec{h}_5=(0,h_1,0),\,\,\vec{g}_5=(0,0,0)$\,\,\,\,\,\,\,\,\,\, and \,\, $\vec{h}_6=(0,0,0),\,\,\vec{g}_6=(0,g_1,0)$.
\end{itemize}
Then, there exists $\delta>0$, such that,  for  any $(u^0,v^0), (u^1,v^1) \in (L^2(0,L))^2$ verifying
$$\|(u^0,v^0)\|_{\X} + \|(u^1,v^1)\|_{\X} \leq \delta,$$
the following assertions are found
\begin{enumerate}
\item[(i)] If $L \in (0,\infty) \setminus \FF_r$, one can find  $\vec{h}_i, \vec{g}_i \in H^{-\frac{1}{3}}(0,T)\times L^2(0,T)\times H^{-\frac{1}{3}}(0,T)$, for $i=1,2$, such that the system \eqref{gg1} with boundary conditions \eqref{gg3} admits a unique solution $(u,v) \in C([0,T];(L^2(0,L))^2)\cap L^2(0,T,(H^1(0,L))^2)$ satisfying \eqref{exact}.
\item[(ii)] For any $L>0$, one can find  $\vec{h}_i, \vec{g}_j \in H^{-\frac{1}{3}}(0,T)\times L^2(0,T)\times H^{-\frac{1}{3}}(0,T)$, for $j=3,4$, such that the system \eqref{gg1} with boundary conditions \eqref{gg3} admits a unique solution $(u,v) \in C([0,T];(L^2(0,L))^2)\cap L^2(0,T,(H^1(0,L))^2)$, satisfying \eqref{exact}.
\item[(iii)]  Let $T>0$ and $L>0$ satisfying
\begin{align*}
1>\frac{\beta C_T}{T}\left[L +\frac{r}{c} \right],
\end{align*}
where $C_T$ is the constant in \eqref{hr4} and $\beta$ is the constant given by the embedding $H^{\frac{1}{3}}(0,T) \subset L ^2(0,T)$. Then, one can find  $\vec{h}_k, \vec{g}_k \in H^{-\frac{1}{3}}(0,T)\times L^2(0,T)\times H^{-\frac{1}{3}}(0,T)$, for $k=5,6$, such that the system \eqref{gg1} with boundary conditions \eqref{gg3} admits a unique solution $(u,v) \in C([0,T];(L^2(0,L))^2)\cap L^2(0,T,(H^1(0,L))^2)$, satisfying \eqref{exact}.
\end{enumerate}
}
\vglue 0.2 cm

Note that Theorem C shows that only one control mechanism is needed to prove the controllability instead of two as in Theorem B. Moreover, the previous theorem reveals that the system \eqref{gg1} is sensitive to changes of boundary conditions, like the KdV equation (see for instance \cite{caicedo_caspistrano_zhang_2015,rosier} and references therein for more details). More precisely, in Theorem C the authors showed that the exact controllability property is derived for any $L>0$ with control functions $h_0, g_0\in H^{-\frac{1}{3}}(0,T)$ and $h_1, g_1\in L^2(0,T)$. However, if we change the position of the controls and consider $h_0(t)=h_2(t)=0$ (resp. $g_0(t)=g_2(t)=0)$ the result with control functions $g_0, g_2\in H^{-\frac{1}{3}}(0,T)$ and $h_1, g_1\in L^2(0,T)$ is obtained if and only if the length $L$ of the spatial domain $(0,L)$ does not belong to a countable set \eqref{critical_f}. In other words, for Neumann boundary condition as in \eqref{gg3}, the \textit{critical length phenomenon} appears. Here, the result was obtained arguing as in \cite{caicedo_caspistrano_zhang_2015,rosier}, i. e., combining the classical duality approach \cite{dolecki,lions} and a fixed point argument. On the other hand, if only one control act on the boundary condition, $h_0(t)=g_0(t)=h_2(t)=g_2(t)=0$ and $g_1(t)=0$ (resp. $h_1(t)=0$), the linearized system is proved to be exactly controllable for small values of the length $L$ and time of control $T$.  In this case, due to some technical difficulties that will become clear during the proof, the observability inequality is proved using multipliers.


\vglue 0.2 cm

Having all these results in hand, a natural question to be asked here is the following one.
\vglue 0.2 cm
\noindent\textbf{Critical Length Phenomenon:} Is there the \textit{critical length phenomenon} to the system \eqref{gg1}-\eqref{gg2}?

\subsection{Main result and notations}
We will consider the system \eqref{gg1} with the following four controls
\begin{equation}\label{gg3_new}
\left\lbrace\begin{tabular}{l l l l}
$u(0,t) = 0$,     & $u(L,t) = 0$,      & $u_{x}(L,t) = h_2(t)$ & in $(0,T)$,\\
$v(0,t) =g_0(t)$, & $v(L,t) = g_1(t)$, & $v_{x}(L,t) = g_2(t)$ & in $(0,T)$.
\end{tabular}\right.
\end{equation}
As conjectured by Capistrano--Filho \textit{et al.} in \cite{capisgallegopazoto2015}, indeed we can prove that  system \eqref{gg1}-\eqref{gg3_new} is controllable if and only if the length $L$ of the spatial domain $(0,L)$ does not belong to a new countable set, i. e.,
\begin{equation}\label{critical_diric}
L \in\!\!\!\!\!/\,\, \FF'_r:= \left\lbrace \pi \sqrt{\frac{(1-a^2b)\alpha(k,l,m,n,s)}{3r}}:k, l, m, n, s \in \N \right\rbrace,
\end{equation}
where
\begin{align*}
\alpha:=\alpha(k,l,m,n,s)=&5k^2+8l^2+9m^2+8n^2+5s^2+8kl+6km\\
+&4kn+2ks+12ml+8ln+3ls+12mn+6ms+8ns.
\end{align*}
Furthermore, it is possible to get the controllability of the system by using only one control
\begin{equation*}
\left\lbrace\begin{tabular}{l l l l}
$u(0,t) = 0$,     & $u(L,t) = 0$,      & $u_{x}(L,t) = h_2(t)$ & in $(0,T)$,\\
$v(0,t) = 0$,     & $v(L,t) = 0$,      & $v_{x}(L,t) = 0$      & in $(0,T)$,
\end{tabular}\right.
\end{equation*}
under the condition
\begin{align}\label{critical_diric-1}
L< \frac{\min\{b,c\}}{  \max\{b,c\} \beta C_T}T,
\end{align}
where $C_T$ is the constant in \eqref{hr4} and $\beta$ is the constant given by the embedding $H^{\frac{1}{3}}(0,T) \subset L ^2(0,T)$.

The analysis describe above is summarized in the main result of the paper, Theorem \ref{main}. In order to make the reading of the proof easier,
throughout the paper we use the following notation for the boundary functions:
\begin{enumerate}
\item[] $\vec{h}_1=(0,0,h_2)$, \ \ $\vec{g}_1=(g_0,g_1,g_2)$ \ \ and \ \ $\vec{h}_2=(0,0,h_2)$, \ \ $\vec{g}_2=(0,0,0)$.
\end{enumerate}
We also introduce the spaces of the boundary functions as follows
$$\HH_T:=H^{\frac{1}{3}}(0,T)\times H^{\frac{1}{3}}(0,T) \times L^2(0,T) \mbox{ and } \ZZ_T:= C([0,T];(L^2(0,L))^2)\cap L^2(0,T,(H^1(0,L))^2)$$
endowed with their natural inner products.  Finally, we consider the space $\mathcal{X}:=(L^2(0,L))^2$ endowed with the inner product
\begin{equation*}
\left\langle (u,v) , (\varphi
,\psi)\right\rangle := \frac{b}{c}\int_0^L u(x)\varphi(x) dx + \int_0^L v(x)\psi(x) dx,\qquad \forall (u,v), (\varphi,\psi) \in \X.
\end{equation*}
With the notation above, we can answer the question mentioned in previous subsection as follows:
\begin{thm}\label{main}
Let $T>0$. Then, there exists $\delta>0$ such that  for  $(u^0,v^0)$, $ (u^1,v^1)$ in $\X$ verifying
$$\|(u^0,v^0)\|_{\X} + \|(u^1,v^1)\|_{\X} \leq \delta,$$
the following hold:
\begin{enumerate}
\item[(i)]  Let $L \in (0,+\infty) \setminus \FF'_r $. Then, one can find  $\vec{h}_1, \vec{g}_1 \in \HH_T$, such that the system \eqref{gg1}-\eqref{gg2} admits a unique solution $(u,v) \in \ZZ_T$ satisfying
\eqref{exact}.
\item[(ii)] Let $L > 0$ satisfying \eqref{critical_diric-1}.
Then, one can find $\vec{h}_2, \vec{g}_2 \in \HH_T$, such that the system \eqref{gg1}-\eqref{gg2} admits a unique solution $(u,v) \in \ZZ_T$, satisfying \eqref{exact}.
\end{enumerate}
 \end{thm}
Theorem \ref{main} will be proved using the same approach that Capistrano--Filho \textit{et al.} used to establish Theorem C. In order to deal with the linearized system, we also use
the classical duality approach \cite{dolecki,lions} which reduces the problem to prove an observability inequality for the solutions of the corresponding adjoint system associated to (\ref{gg1})-(\ref{gg2}):
\begin{equation}\label{int_linadj}
\begin{cases}
\varphi_t + \varphi_{xxx} + \frac{ab}{c}\psi_{xxx}=0,& \text{in}\,\, (0,L)\times (0,T), \\
\psi_t  +\frac{r}{c}\psi_x+a\varphi_{xxx} +\frac{1}{c}\psi_{xxx} =0,  & \text{in}\,\, (0,L)\times (0,T), \\
\varphi(0,t)=\varphi(L,t)=\varphi_x(0,t)=0, & \text{in}\,\, (0,T),\\
\psi(0,t)=\psi(L,t)=\psi_x(0,t)=0, & \text{in}\,\, (0,T), \\
\varphi(x,T)= \varphi^1(x), \ \ \psi(x,T)= \psi^1(x),  & \text{in}\,\, (0,L).
\end{cases}
\end{equation}
Similarly, as in \cite{capisgallegopazoto2015}, one will encounter some difficulties that demand special attention. To prove assertion $(i)$ we need to prove a hidden regularity for the solutions of the system of the linear system \eqref{int_linadj}. In our case, the result is given by the following lemma.
\begin{lemma}[Kato sharp trace regularities]\label{sharp_int}
For any $(\varphi^0, \psi^0) \in \X$, the system \eqref{int_linadj} admits a unique solution $(\varphi, \psi) \in \ZZ_T$, such that it possess the following sharp trace properties
\begin{equation}\label{hr4_int}
\underset{0\leq x \leq L}{\sup} \|  (\partial^k_x \varphi(x,\cdot),\partial^k_x \psi(x,\cdot))\|_{(H^{\frac{1-k}3}(0,T))^2}\le C_T\|(\varphi^0,\psi^0)\|_{(L^2(0,L))^2}, \quad \text{for $k=0,1,2.$}
\end{equation}
\end{lemma}
The sharp Kato smoothing properties  of solutions of the Cauchy problem of the KdV  equation posed on the whole line $\mathbb{R}$ due to Kenig, Ponce and Vega \cite{kenigponcevega1991} will play an important role in the  proof of  Lemma \ref{sharp_int}. In what concerns the assertion $(ii)$, the observability inequality for the solutions of \eqref{int_linadj} is proved using multipliers  together with the Lemma \ref{sharp_int}. It is precisely the hidden regularity (sharp trace regularity) given by Lemma \ref{sharp_int} that enable us to prove Theorem B with less controls.

The program of this work was carried out for the particular choice of boundary control inputs
and aims to establish as a fact that such a model predicts the interesting qualitative properties initially
observed for the KdV equation. Consideration of this issue for nonlinear dispersive equations has received considerable attention, specially the problems
related to the study of the controllability properties.

\medskip

The plan of the present paper is as follows.

\medskip

---- In Section 2,  we show that the linear system associated to \eqref{gg1}-\eqref{gg2} is global well-posedness in $\ZZ_T$. Additionally, we present various estimates, among them Lemma \ref{sharp_int} for the solution of the adjoint system.

\medskip

----  Section 3 is intended to show the controllability of the linear system associated with \eqref{gg1} when four controls are considered in the boundary conditions. Moreover, when only one function is a control input the boundary controllability result is also proved. Here, the hidden regularities for the solutions of the adjoint system  presented in the Section 2 are used to prove observability inequalities associated to the control problem.

\medskip

---- In Section 4, we prove the local well-posedness of the system \eqref{gg1}-\eqref{gg2} in $\ZZ_T$. After that, the exact boundary controllability of the nonlinear system is proved \textit{via} contraction mapping principle.

\medskip

---- Finally, Section 5 contains some remarks and related problems.

\section{Well-posedness}
\subsection{Linear homogeneous system}
Firstly, we establish the well-posedness of the initial-value problem of the linear  system associated to (\ref{gg1})-(\ref{gg2}):
\begin{equation}\label{gglin}
\left\lbrace \begin{tabular}{l l}
$u_t + u_{xxx} + av_{xxx}  =0$, & in $(0,L)\times (0,T)$,\\
$v_t +\frac{r}{c}v_x+\frac{ab}{c}u_{xxx} + \frac{1}{c}v_{xxx} =0$, & in $(0,L)\times (0,T)$,\\
$u(0,t) =u(L,t) =u_{x}(L,t) = 0$,& in $(0,T)$,\\
$v(0,t) =v(L,t) =v_{x}(L,t) = 0$,& in $(0,T)$,\\
$u(x,0)=u^0(x), \quad v(x,0) = v^0(x)$, & in $(0,L)$.
\end{tabular}\right.
\end{equation}
Let us define the operator $A$ by
\begin{equation}
\label{eq:def-A} \displaystyle A \left( \begin{array}{c} u \\ v
\end{array}
\right) = -\left(
\begin{array}{cc}
\displaystyle \partial_{xxx} & a\partial_{xxx} \\
\displaystyle  \frac{ab}{c}\partial_{xxx}   &\frac{r}{c} \partial_{x} +\frac{1}{c} \partial_{xxx}
\end{array}\right)    \left(
\begin{array}{c}
u\\
v
\end{array}
\right)\quad
\end{equation}
with domain
$$D(A)=\left\{(u,v) \in (H^3(0,L))^2\,:
u(0)=v(0)=u(L)=v(L)=u_{x}(L)=v_{x}(L)=0\right\}\subset \X.$$
The linear system \eqref{gglin} can be written in abstract form  as
\begin{equation}\label{cauchy}
\begin{cases}
U_t=AU,\\
U(0)=U_0,
\end{cases}
\end{equation}
where $U:=(u,v)$ and $U_0:=(u^0,v^0)$.  We denote by  $A^*$ the adjoint operator of $A$, defined by
\begin{equation}\label{Adjunto}
 \displaystyle A^*\left( \begin{array}{cc} \varphi\\ \psi\end{array}\right) = \left(
\begin{array}{cc}
\displaystyle \partial_{xxx} &  \frac{ab}{c}\partial_{xxx} \\ \displaystyle
a\partial_{xxx} &  \displaystyle \frac{r}{c} \partial_{x} + \displaystyle\frac{1}{c} \partial_{xxx}
\end{array}
\right)
\left( \begin{array}{c} \varphi \\ \psi
\end{array}
\right)
\end{equation}
with domain
$$D(A^*)=\left\{(\varphi,\psi)\in (H^3(0,L))^2\,:\,
\varphi(0)=\psi(0)=\varphi(L)=\psi(L)=\varphi_{x}(0)=\psi_{x}(0)=0\right\}\subset \X.$$
The following results can be found in \cite{micuortegapazoto2009}.
\begin{prop} \label{A-disipativo}
The operator $A$ and its adjoint $A^*$ are dissipative in $\X$.
\end{prop}

As a consequence, we have that (see Corol. 4.4, page 15, in \cite{pazy}):

\begin{thm}\label{homog}
Let  $U_0\in \X$. There exists a unique (weak) solution $U=S(\,\cdot\,)U_0$ of \eqref{gglin} such that
\begin{equation}\label{sol}
U \in C\left([0,T];\X\right)\cap H^1\left(0,T;
(H^{-2}(0,L))^2\right).
\end{equation}
Moreover, if $U_0\in D(A)$ then \eqref{gglin} has a  unique (classical) solution $U$ such that
\begin{equation*}
U\in C([0,T];D(A))\cap C^1((0,T);\X).
\end{equation*}
\end{thm}

The next result reveals a gain of regularity for the weak solutions given by Theorem \ref{homog}.

\begin{thm}
Let $(u^0,v^0)$ in  $\X$ and $(u,v)$ the weak solution of \eqref{gglin}. Then, $$(u,v) \in L^2(0,T; (H^1(0,L))^2)$$ and there exists a positive constant $c_0$ such that
\begin{equation*}
\|(u,v)\|_{ L^2(0,T; (H^1(0,L))^2)} \leq c_0 \|(u^0,v^0)\|_{\X}.
\end{equation*}
Moreover, there exist two positive constants $c_1$ and $c_2$ such that
\begin{equation*}
\|(u_x(0,\cdot),v_x(0,\cdot))\|_{\X}^2 \leq c_1 \|(u^0,v^0)\|_{\X}^2.
\end{equation*}
and
\begin{equation*}
\|(u^0,v^0)\|_{\X}^2 \leq \frac{1}{T}\|(u,v)\|_{ L^2(0,T; \X)}^2+c_2\|(u_x(0,\cdot),v_x(0,\cdot))\|_{\X}^2.
\end{equation*}
\end{thm}

\subsection{Linear nonhomogeneous system}

In this subsection, we study the nonhomogeneous system corresponding to \eqref{gg1}-\eqref{gg2}:

\begin{equation}\label{gglin2}
\begin{cases}
u_t + u_{xxx} + a v_{xxx} =0, & \text{in} \,\, (0,L)\times (0,T),\\
v_t +\frac{r}{c} v_x +\frac{ab}{c} u_{xxx} +\frac{1}{c} v_{xxx}  =0,  & \text{in} \,\, (0,L)\times (0,T), \\
u(0,t)=h_0(t),\,\,u(L,t)=h_1(t),\,\,u_{x}(L,t)=h_2(t), & \text{in} \,\, (0,T),\\
v(0,t)=g_0(t),\,\,v(L,t)=g_1(t),\,\,v_{x}(L,t)=g_2(t), & \text{in} \,\, (0,T),\\
u(x,0)= u^0(x), \quad v(x,0)=  v^0(x), & \text{in} \,\, (0,L).\\
\end{cases}
\end{equation}

The next well-posedness result can be found in \cite[Theorems 2.3, 2.4]{micuortegapazoto2009}.

\begin{thm}\label{mop_1}
There exists a unique linear and continuous map
$$\Psi : \X \times (H^1_0(0,T))^2 \times (H^1_0(0,T))^2\times \X \rightarrow C([0,T];\X)\cap L^2(0,T;(H^1(0,L))^2)$$
such that, for any $(u^0,v^0)$ in $D(A)$ and $h_i,g_i$ in $C_0^2[0,T]$, with $i=0,1,2$,
$$\Psi((u^0,v^0),(h_0,g_0,h_1,g_1,h_2,g_2))=(u,v)$$
where $(u,v)$ is the unique classical solution of \eqref{gglin2}. Moreover, there exists a positive constant $C>0$ such that
\begin{equation*}
\|(u,v)\|_{C([0,T];\X)}^2+\|(u,v)\|_{L^2(0,T;(H^1(0,L))^2)} \leq C \left[\|(u^0,v^0)\|_{\X}^2+\sum_{i=0}^{2}(\|h_i\|_{H^1(0,T)}+\|g_i\|_{H^1(0,T)})\right].
\end{equation*}
\end{thm}
Our main goal in this subsection is to improve Theorem \ref{mop_1}. We will obtain some important trace estimates, using a new tool, which reveals the sharp Kato smoothing (or hidden regularity) for the solution of system \eqref{gglin2}. In order to do that, we consider the system
\begin{equation}\label{gglin3}
\left\lbrace \begin{tabular}{l l}
$u_t + u_{xxx} + av_{xxx}  =f$, & in $(0,L)\times (0,T)$,\\
$v_t +\frac{ab}{c}u_{xxx} + \frac{1}{c}v_{xxx} =s$, & in $(0,L)\times (0,T)$,\\
$u(0,t) = h_0(t),\,\,u(L,t) = h_1(t),\,\,u_{x}(L,t) = h_2(t)$,& in $(0,T)$,\\
$v(0,t) =g_0(t),\,\,v(L,t) = g_1(t),\,\, v_{x}(L,t) = g_2(t)$,& in $(0,T)$,\\
$u(x,0)=u^0(x), \quad v(x,0) = v^0(x)$, & in $(0,L)$,
\end{tabular}\right.
\end{equation}
where $f=f(x,t)$ and $s=s(x,t)$. Then, we have the following result:
\begin{prop}\label{prop1}
Let $T>0$ be given, for any $(u^0,v^0)$ in $\X$, $f,s$ in $L^1(0,T;L^2(0,L))$ and $\overrightarrow{h}:=(h_0,h_1,h_2)$, $\overrightarrow{g}:=(g_0,g_1,g_2)$ in $\HH_T$, the IBVP (\ref{gglin3}) admits a unique solution $(u,v) \in \ZZ_T$, with
\begin{equation}\label{hr1}
\partial_x^k u,\partial_x^k v \in L^{\infty}_x(0,L;H^{\frac{1-k}{3}}(0,T)),  \quad k=0,1,2.
\end{equation}
Moreover, there exist $C>0$, such that
\begin{multline}\label{hr2}
\|(u,v)\|_{\ZZ_T}+\sum_{k=0}^{2}\|(\partial^k_xu,\partial^k_xv)\|_{L^{\infty}_x(0,L;H^{\frac{1-k}{3}}(0,T))}\leq C\left\lbrace \|(u^0,v^0)\|_{(L^2(0,L))^2}+\|(\overrightarrow{h},\overrightarrow{g})\|_{\HH_T} \right. \\
\left.+\|(f,s)\|_{L^1(0,T;L^2(0,L))} \right\rbrace.
\end{multline}
\end{prop}
To prove the Proposition \ref{prop1}, we need an auxiliary result.
\begin{prop}\label{prop2}
Consider the following nonhomogeneous Korteweq-de Vries equation
\begin{equation}
\begin{cases}\label{kdv}
u_t+ \alpha u_{xxx}=f,& \text{in} \,\, (0,L)\times (0,T), \\
u(0,t) = h_0(t),\,\,u(L,t) = h_1(t),\,\,u_{x}(L,t) = h_2(t), & \text{in} \,\, (0,T),\\
u(x,0)=u^0(x), & \text{in} \,\, (0,L).
\end{cases}
\end{equation}
For any $u^0 \in L^2(0,L)$, $f \in L^1(0,T;L^2(0,L))$,  $\overrightarrow{h}:=(h_0,h_1,h_2) \in \HH_T$ and $\alpha \in \R$,  the IBVP (\ref{kdv}) admits a unique solution
\begin{equation*}
u \in \X_T:=C([0,T]; L^2(0,L))\cap L^2(0,T; H^1(0,L))
\end{equation*}
with
\begin{equation}\label{hr1'}
\partial_x^k u \in L^{\infty}_x(0,L;H^{\frac{1-k}{3}}(0,T)),  \quad k=0,1,2.
\end{equation}
Moreover, there exist $C>0$, such that
\begin{equation}\label{hr2'}
\|u\|_{\X_T}+\sum_{k=0}^{ 2 }\|\partial^k_xu\|_{L^{\infty}_x(0,L;H^{\frac{1-k}{3}}(0,T))} \leq C\left\lbrace \|u^0\|_{L^2(0,L)}+\|(\overrightarrow{h},\overrightarrow{g})\|_{\HH_T}+\|(f,s)\|_{L^1(0,T;L^2(0,L))} \right\rbrace.
\end{equation}
\end{prop}

\begin{proof}
 When $k=0,1$ the result was proved by Bona, Sun and Zhang in \cite{bonasunzhang2003}. Therefore, for the sake of completeness, we prove the result for the case when $k=2$.

Proceeding as in \cite{bonasunzhang2003}, it is sufficient to prove that the solution $v$ of the following linear non-homogeneous boundary value problem,
\begin{equation}\label{a1}
\begin{cases}
v_t+  av_{xxx}=0,& \text{in} \,\, (0,L)\times (0,T), \\
v(0,t) = h_0(t),\,\,v(L,t) = h_1(t),\,\,v_{x}(L,t) = h_2(t), & \text{in} \,\, (0,T),\\
v(x,0)=0, & \text{in} \,\, (0,L).
\end{cases}
\end{equation}
satisfies
\begin{equation}\label{xxtrace}
\sup_{0 \leq x \leq L }\|\partial_x^2 v(x,\cdot)\|_{H^{-\frac{1}{3}}(0,T)}\leq C_T \|\overrightarrow{h}\|_{\HH_T}.
\end{equation}
Indeed, applying the Laplace transform with respect to $t$, \eqref{a1} is converted to
\begin{equation}\label{a2}
\begin{cases}
s \hat{v}(x,s)+  a\hat{v}_{xxx}(x,s)=0,& \text{in} \,\, (0,L)\times (0,T), \\
\hat{v}(0,s) =\hat{h}_0(s),\,\,\hat{v}(L,s) = \hat{h}_1(s),\,\,\hat{v}_{x}(L,s) = \hat{h}_2(s), & \text{in} \,\, (0,T),\\
\hat{v}(x,0)=0, & \text{in} \,\, (0,L).
\end{cases}
\end{equation}
where
\begin{equation*}
\hat{v}(x,s)=\int_{0}^{\infty} e^{-st}v(x,t)dt \quad \text{and} \quad
\hat{h}_j(s)=\int_{0}^{\infty} e^{-st}h_j(t)dt, \qquad j=0,1,2.
\end{equation*}
The solution $\hat{v}(x,s)$ can be written in the form  $\hat{v}(x,s)= \sum_{j=0}^{2}c_j(s)e^{-\lambda_j(s)x}$, where $\lambda_j(s)$ are the solutions of the characteristic equation $s+a\lambda^3=0$ and $c_j(s)$, solve the linear system
\begin{equation*}
\left(
\begin{array}{ccc}
 1&1&1 \\
e^{\lambda_0}&  e^{\lambda_1}&  e^{\lambda_2}\\
\lambda_0 e^{\lambda_0}& \lambda_1 e^{\lambda_1}& \lambda_2 e^{\lambda_2} \\
\end{array} \right)
\left(\begin{array}{ccc}
c_0 \\ c_1 \\c_2  \end{array}\right) =
\left(\begin{array}{ccc}
\hat{h}_0\\ \hat{h}_1\\ \hat{h}_2 \end{array}\right).
\end{equation*}
Using the Cramer rule, we obtain $c_j(s)=\frac{\Delta_j(s)}{\Delta(s)}$, $j=0,1,2$, where $\Delta(s)$ is the determinant of the coefficient matrix and $\Delta_j(s)$ the determinants of the matrices that are obtained by replacing the ith-column by the column vector $\overrightarrow{h}:=(\hat{h}_0(s),\hat{h}_1(s),\hat{h}_2(s))$. Taking the inverse Laplace transform of $\hat{v}$, yields
\begin{equation*}
v(x,t)=\frac{1}{2\pi i}\sum_{j=0}^{2}\int_{r-i\infty}^{r+i\infty}e^{st}\frac{\Delta_j(s)}{\Delta(s)}e^{\lambda_j(s)x}ds
\end{equation*}
for any $r>0$. Note that, $v$ may also be written in the form
\begin{equation}\label{aa1}
v(x,t)=\sum_{m=0}^2v_m(x,t),
\end{equation}
where $v_m(x, t)$ solves \eqref{a1} with $h_j = 0$ when $j \neq m$, $m, j = 0, 1, 2$. Thus, $v_m$ take the form
\begin{equation}\label{rep}
v_m(x,t)=\frac{1}{2\pi i}\sum_{j=0}^{2}\int_{r-i\infty}^{r+i\infty}e^{st}\frac{\Delta_{j,m}(s)}{\Delta(s)}e^{\lambda_j(s)x}\hat{h}_m(s)ds:=[W_{m}(t)h_m(t)](x)
\end{equation}
where  $\Delta_{j,m}(s)$ is obtained from $\Delta_j(s)$ by letting $\hat{h}_m \equiv 1$ and $\hat{h}_j\equiv 0$, for $j \neq m$, $j,m=0,1,2$. Moreover, note that, the right-hand sides are continuous with respect to $r$ for $r\geq 0$. As the left-hand sides do not depend on $r$, it follows that we may take $r=0$. Thus, we can write $v_m$ as
\begin{equation}\label{a4}
v_m(x,t)=v_m^{+}(x,t)+v_m^{-}(x,t),
\end{equation}
where
\begin{gather*}
v_m^+(x,t)=\frac{1}{2\pi i}\sum_{j=0}^{2}\int_{0}^{i\infty}e^{st}\frac{\Delta_{j,m}(s)}{\Delta(s)}e^{\lambda_j(s)x}\hat{h}_m(s)ds, \\
v_m^-(x,t)=\frac{1}{2\pi i}\sum_{j=0}^{2}\int^{0}_{-i\infty}e^{st}\frac{\Delta_{j,m}(s)}{\Delta(s)}e^{\lambda_j(s)x}\hat{h}_m(s)ds.
\end{gather*}
Making the substitution $s = ia \rho^3L^3$ with $\rho \geq 0$ in the characteristic equation, the three roots are given in terms of $\rho$ by
\begin{equation*}
\lambda_0(\rho)=iL\rho, \quad \lambda_1(\rho)=-iL\rho\left(\frac{1+i\sqrt{3}}{2}\right), \quad \lambda_2(\rho)=-iL\rho\left(\frac{1-i\sqrt{3}}{2}\right).
\end{equation*}
Thus, $v_m^+$ and $v_m^-$ have the following representation,
\begin{equation}\label{a5}
v_m^+(x,t)=\frac{3aL^3}{2\pi}\sum_{j=0}^{2}\int_{0}^{\infty}e^{ia\rho^3L^3t}\frac{\Delta_{j,m}^+(\rho)}{\Delta^+(\rho)}e^{\lambda_j^+(\rho)x}\hat{h}^+_m(\rho)\rho^2d\rho\quad \text{and} \quad v_m^-(x,t)=\overline{v_m^+(x,t)},
\end{equation}
where $\Delta_{j,m}^+(\rho)=\Delta_{j,m}(ia \rho^3L^3)$, $\Delta^+(\rho)=\Delta(ia \rho^3L^3)$, $\lambda_j^+(\rho)=\lambda_j(ia \rho^3L^3)$ and $\hat{h}^+_m(\rho)=\hat{h}_m(ia \rho^3L^3)$. Thus, we have
\begin{align*}
\partial_x^2 v_m^+(x,t)&= \frac{3aL^3}{2\pi}\sum_{j=0}^{2}\int_{0}^{\infty}e^{ia\rho^3L^3t}(\lambda_j^+(\rho))^2\frac{\Delta_{j,m}^+(\rho)}{\Delta^+(\rho)}e^{\lambda_j^+(\rho)x}\hat{h}^+_m(\rho)\rho^2d\rho \\
&=\frac{1}{2\pi}\sum_{j=0}^{2}\int_{0}^{\infty}e^{i\mu t}(\lambda_j^+(\theta(\mu)))^2\frac{\Delta_{j,m}^+(\theta(\mu))}{\Delta^+(\theta(\mu))}e^{\lambda_j^+(\theta(\mu))x}\hat{h}^+_m(\theta(\mu))d\mu,
\end{align*}
where $\theta(\mu)$ is the real solution of $\mu=a\rho^3L^3$, for $\rho \geq 0$. Here
\begin{equation*}
\lambda_0(\rho)=iL\rho, \quad \lambda_1(\rho)=-iL\rho\left(\frac{1+i\sqrt{3}}{2}\right), \quad \lambda_2(\rho)=-iL\rho\left(\frac{1-i\sqrt{3}}{2}\right).
\end{equation*}
Applying Plancherel Theorem (with respect to $t$), yields for any $x\in (0,L)$,
\begin{align*}
\|\partial_x^2 v_m^+(x,\cdot)\|^2_{H^{-\frac{1}{3}}(0,T)} &\leq \frac{1}{2\pi}\sum_{j=0}^{2}\int_{0}^{\infty} |\mu|^{-\frac{2}{3}} \left|(\lambda_j^+(\theta(\mu)))^2\frac{\Delta_{j,m}^+(\theta(\mu))}{\Delta^+(\theta(\mu))} e^{\lambda_j^+(\theta(\mu))x} \right|^2\left|\hat{h}_m^+(\theta(\mu))\right|^2 d\mu \\
&=\frac{1}{2\pi}\sum_{j=0}^{2}\int_{0}^{\infty} a^{-\frac23}\rho^{-2}L^{-2} \left|(\lambda_j^+(\rho))^2\frac{\Delta_{j,m}^+(\rho)}{\Delta^+(\rho)} e^{\lambda_j^+(\rho)x} \right|^2\left|\hat{h}_m^+(\rho)\right|^2(3aL^3\rho^2)d\rho, \\
&=\frac{3a^{-\frac13}L}{2\pi}\sum_{j=0}^{2}\int_{0}^{\infty} \left|(\lambda_j^+(\rho))^2\frac{\Delta_{j,m}^+(\rho)}{\Delta^+(\rho)} e^{\lambda_j^+(\rho)x} \right|^2\left|\hat{h}_m^+(\rho)\right|^2d\rho.
\end{align*}
On the other hand, note that
\begin{equation*}
\sup_{0\leq x\leq L} \left| e^{\lambda_0^+(\rho)x}\right| \leq C, \quad \sup_{0\leq x\leq L} \left| e^{\lambda_1^+(\rho)x}\right| \leq Ce^{\frac{\sqrt{3}}{2}\rho L}, \quad \sup_{0\leq x\leq L} \left| e^{\lambda_2^+(\rho)x}\right| \leq Ce^{-\frac{\sqrt{3}}{2}\rho L}.
\end{equation*}
Then, it follows that
\begin{align*}
\|\partial_x^2 v_m^+(x,\cdot)\|^2_{H^{-\frac{1}{3}}(0,T)} &\leq C\left\lbrace \int_{0}^{\infty} \rho^4\left|\frac{\Delta_{0,m}^+(\rho)}{\Delta^+(\rho)}  \right|^2\left|\hat{h}_m^+(\rho)\right|^2d\rho   +   \int_{0}^{\infty} \rho^4 e^{\sqrt{3}\rho L} \left|\frac{\Delta_{1,m}^+(\rho)}{\Delta^+(\rho)}  \right|^2\left|\hat{h}_m^+(\rho)\right|^2d\rho \right. \\
&\left. +   \int_{0}^{\infty} \rho^4e^{-\sqrt{3}\rho L}\left|\frac{\Delta_{2,m}^+(\rho)}{\Delta^+(\rho)}  \right|^2\left|\hat{h}_m^+(\rho)\right|^2d\rho\right\rbrace.
\end{align*}
Using the estimates of $\left| \frac{\Delta^+_{j,m}(\rho)}{\Delta^+(\rho)}\right|$ proved in \cite{bonasunzhang2003}, that is,
\begin{equation}\label{asymp1}
\begin{array}{|l|l|l|}
\hline
\dfrac{\Delta_{0,0}^+(\rho)}{\Delta^+(\rho)} \sim e^{-\frac{\sqrt{3}}{2}\rho L} & \dfrac{\Delta_{1,0}^+(\rho)}{\Delta^+(\rho)}  \sim e^{-\sqrt{3}\rho L} & \dfrac{\Delta_{2,0}^+(\rho)}{\Delta^+(\rho)}  \sim 1 \\
\hline
\dfrac{\Delta_{0,1}^+(\rho)}{\Delta^+(\rho)}  \sim 1 & \dfrac{\Delta_{1,1}^+(\rho)}{\Delta^+(\rho)}  \sim e^{-\frac{\sqrt{3}}{2}\rho L} & \dfrac{\Delta_{2,1}^+(\rho)}{\Delta^+(\rho)}  \sim 1\\
\hline
\dfrac{\Delta_{0,2}^+(\rho)}{\Delta^+(\rho)}  \sim \rho^{-1} & \dfrac{\Delta_{1,2}^+(\rho)}{\Delta^+(\rho)}  \sim   \rho^{-1} e^{-\frac{\sqrt{3}}{2}\rho L}  & \dfrac{\Delta_{2,2}^+(\rho)}{\Delta^+(\rho)}  \sim \rho^{-1} \\
\hline
\end{array}
\end{equation}
we obtain
\begin{align*}
\|\partial_x^2 v_0^+(x,\cdot)\|^2_{H^{-\frac{1}{3}}(0,T)} &\leq C \int_{0}^{\infty} \rho^4\left|\hat{h}_0^+(\rho)\right|^2d\rho =C \int_{0}^{\infty} \rho^4\left|\hat{h}_0(ia\rho^3L^3)\right|^2d\rho \\
&= C \int_{0}^{\infty} \rho^4\left| \int_0^{\infty} e^{-ia\rho^3 L^3 t} h_0(t)dt\right|^2d\rho.
\end{align*}
Setting $\mu=a\rho^3 L^3$, it follows that
\begin{align*}
\|\partial_x^2 v_0^+(x,\cdot)\|^2_{H^{-\frac{1}{3}}(0,T)} &= C \int_{0}^{\infty} \rho^4\left| \int_0^{\infty} e^{-ia\rho^3 L^3 t} h_0(t)dt\right|^2d\rho
&\leq C \int_{0}^{\infty} \mu^{\frac{2}{3}}\left| \int_0^{\infty} e^{-i\mu t} h_0(t)dt\right|^2d\mu \\
&\leq C \|h_0\|_{H^{\frac{1}{3}}(\R^+)}^2.
\end{align*}
Similarly, we obtain estimates for $\partial_x^2 v_1$ and $\partial_x^2 v_2$ in $H^{-\frac{1}{3}}(0,T)$. Indeed,
\begin{align*}
\|\partial_x^2 v_1^+(x,\cdot)\|^2_{H^{-\frac{1}{3}}(0,T)}
&\leq C \|h_1\|_{H^{\frac{1}{3}}(\R^+)}^2
\end{align*}
and
\begin{align*}
\|\partial_x^2 v_2^+(x,\cdot)\|^2_{H^{-\frac{1}{3}}(0,T)}
&\leq C \int_{0}^{\infty} \rho^2\left|\hat{h}_1^+(\rho)\right|^2d\rho \leq C \|h_2\|_{L^2(\R^+)}^2.
\end{align*}
Thus, \eqref{xxtrace} follows from \eqref{aa1}, \eqref{a4} and \eqref{a5}. We also observe that, as in \cite[Theorem 2.10]{bonasunzhang2003},
the solutions can be written in the form of the boundary integral operator $W_{bdr}$ as follows
\begin{equation}\label{rep1}
v(x,t)=[W_{bdr}\overrightarrow{h}](x,t)=\sum_{i=0}^2 [W_j(t)h_j](x),
\end{equation}
where $W_j$ is defined in \eqref{rep}.
\end{proof}

\begin{proof}[\textbf{Proof of Proposition \ref{prop1}.}]
Consider the change of variable
\begin{equation}\label{cv1}
\left\lbrace\begin{tabular}{l}
$u = 2a \widetilde{u} +2a \widetilde{v}$, \\
$v = \left(\left(\frac{1}{c}-1\right)+\lambda\right) \widetilde{u}+ \left(\left(\frac{1}{c}-1\right)-\lambda\right) \widetilde{v}$
\end{tabular}\right.
\end{equation}
with $\lambda=\sqrt{\left(\frac{1}{c}-1\right)^2+\frac{4a^2b}{c}}$. Thus, we can transform the linear system (\ref{gglin3}) into
\begin{equation}\label{kdvln1}
\left\{
\begin{array}{ll}\vspace{2mm}
\widetilde{u}_t + \alpha_{-}\widetilde{u}_{xxx} =\widetilde{f},\\
 \widetilde{v}_t +\alpha_{+}\widetilde{v}_{xxx}=\widetilde{s},  \\
\widetilde{u}(0,t) = \widetilde{h}_0(t), \  \widetilde{u}(L,t) = \widetilde{h}_1(t),  \  \widetilde{u}_{x}(L,t)  = \bar{h}_2(t) , \\
\widetilde{v}(0,t) = \widetilde{g}_0(t), \ \widetilde{v}(L,t)  =  \widetilde{g}_1(t), \ \widetilde{v}_{x}(L,t)  = \widetilde{g}_2(t) , \\
 \widetilde{u}(x,0)= \widetilde{u}^0(x), \quad \widetilde{v}(x,0)  = \widetilde{v}^0(x),
\end{array}
\right.
\end{equation}
where $\alpha_{\pm} = -\frac{1}{2}\left(\left(\frac{1}{c}-1\right)\pm \lambda\right)$ and
\begin{equation*}
\left\lbrace\begin{tabular}{l l l l}
$\widetilde{f}=-\frac{1}{2}\left(\frac{\alpha_+}{a\lambda}f+\frac{1}{\lambda}s\right)$, & $\widetilde{u}_0=-\frac{1}{2}\left(\frac{\alpha_{-}}{a\lambda}u^0-\frac{1}{\lambda}v^0\right)$, & $\widetilde{h}_i=-\frac{1}{2}\left(\frac{\alpha_{-}}{a\lambda}h_i-\frac{1}{\lambda}g_i\right),$ & $i=0,1,2,$ \\
\\
$\widetilde{s}=-\frac{1}{2}\left(\frac{\alpha_{-}}{a\lambda}f-\frac{1}{\lambda}s\right),$ & $\widetilde{v}_0=\frac{1}{2}\left(\frac{\alpha_{+}}{a\lambda}u^0-\frac{1}{\lambda}v^0\right)$, & $\widetilde{g}_i=\frac{1}{2}\left(\frac{\alpha_{+}}{a\lambda}h_i-\frac{1}{\lambda}g_i\right),$  & $i=0,1,2.$
\end{tabular}\right.
\end{equation*}
The system (\ref{kdvln1}) can be decouple into two KdV equations as follows:
\begin{equation}\label{chanvar}
\left\{
\begin{array}{ll}\vspace{2mm}
\widetilde{u}_t + \alpha_{-}\widetilde{u}_{xxx} =\widetilde{f},\\
\widetilde{u}(0,t) = \widetilde{h}_0(t), \  \widetilde{u}(L,t) = \widetilde{h}_1(t),  \\
\widetilde{u}_x(L,t)  = \widetilde{h}_2(t),\\
\widetilde{u}(0,x)= \widetilde{u}^0(x)
\end{array}
\right. \quad  \text{and}  \quad
\left\{
\begin{array}{ll}\vspace{2mm}
\widetilde{v}_t +\alpha_{+}\widetilde{v}_{xxx}=\widetilde{s},  \\
\widetilde{v}(0,t) = \widetilde{g}_0(t), \ \widetilde{v}(L,t)  =  \widetilde{g}_1(t), \\
\widetilde{v}_{x}(L,t)  = \widetilde{g}_2(t), \\
\widetilde{v}(x,0)  = \widetilde{v}^0(x).
\end{array}
\right.
\end{equation}
Note that for $\alpha_{\pm}$ to be nonzero, it is sufficient to assume that $a^2b \neq 1$. Then, it is easy to see that
\begin{equation*}
(\widetilde{u}^0, \widetilde{v}^0) \in \X,\quad (\widetilde{f},\widetilde{s}) \in L^1(0,T;(L^2(0,L))^2), \quad \overrightarrow{\widetilde{h}}, \overrightarrow{\widetilde{g}} \in \HH_T.
\end{equation*}
By Proposition  \ref{prop2}, we obtain the existence of $(\widetilde{u},\widetilde{v})$, solution of the system (\ref{chanvar}) belongs to $\ZZ_T$, such that
\begin{equation*}
\partial_x^k \widetilde{u},\partial_x^k \widetilde{v} \in L^{\infty}_x(0,L;H^{\frac{1-k}{3}}(0,T)),  \quad k=0,1,2
\end{equation*}
and
\begin{multline*}
\|(\widetilde{u},\widetilde{v})\|_{\ZZ_T}+\sum_{k=0}^{2}\|(\partial^k_x\widetilde{u},\partial^k_x\widetilde{v})\|_{L^{\infty}_x(0,L;H^{\frac{1-k}{3}}(0,T))}\leq C\left\lbrace \|(\widetilde{u}^0,\widetilde{v}^0)\|_{(L^2(0,L))^2}+\|(\overrightarrow{\widetilde{h}},\overrightarrow{\widetilde{g}})\|_{\HH_T} \right. \\
\left.+\|(\widetilde{f},\widetilde{s})\|_{L^1(0,T;L^2(0,L))} \right\rbrace.
\end{multline*}
Furthermore, as in \cite{bonasunzhang2003}, we can write $\widetilde{u}$ and $\widetilde{v}$  in its integral form:
\begin{equation*}
\widetilde{u}(t)=W_0^{-}(t)\widetilde{u}^0+W_{bdr}^{-}(t)\overrightarrow{\widetilde{h}} + \int_0^tW_0^{-}(t-\tau)\widetilde{f}(\tau)d\tau,
\end{equation*}
\begin{equation*}
\widetilde{v}(t)=W_0^{+}(t)\widetilde{v}^0+W_{bdr}^{+}(t)\overrightarrow{\widetilde{g}} + \int_0^tW_0^{+}(t-\tau)\widetilde{s}(\tau)d\tau,
\end{equation*}
where $\{W_0^{\pm}(t)\}_{t\geq 0}$ is the $C_0$-semigroup in the space $L^2(0,L)$ generated by the linear operator $$A^{\pm}=-\alpha_{\pm}g''',$$ with domain $$D(A^{\pm})=\{g \in H^3(0,L): g(0)=g(L)=g'(L)=0\},$$
and $\{W_{bdr}^{\pm}(t)\}_{t\geq 0}$ is the operator given in \eqref{rep1}. By using the change of variable, it is easy to see that
\begin{equation*}
\begin{cases}
u(t)=W_0^{-}(t)u^0+W_{bdr}^{-}(t)\overrightarrow{h} + \int_0^tW_0^{-}(t-\tau)f(\tau)d\tau, \\
v(t)=W_0^{+}(t)v^0+W_{bdr}^{+}(t)\overrightarrow{g} + \int_0^tW_0^{+}(t-\tau)s(\tau)d\tau.
\end{cases}
\end{equation*}
Therefore, the prove is complete.
\end{proof}

By using standard fixed point argument together with Propositions \ref{prop1} and \ref{prop2} we show the global well-posedness of the  system (\ref{gglin2}).

\begin{thm}\label{teo1}
Let $T>0$ be given. For any $(u^0,v^0)$ in $\X$ and $\overrightarrow{h}:=(h_0,h_1,h_2)$, $\overrightarrow{g}:=(g_0,g_1,g_2)$ in $\HH_T$, the IBVP (\ref{gglin2}) admits a unique solution $(u,v) \in \ZZ_T$, with
\begin{equation*}
\partial_x^k u,\partial_x^k v \in L^{\infty}_x(0,L;H^{\frac{1-k}{3}}(0,T)),  \quad k=0,1,2.
\end{equation*}
Moreover, there exist $C>0$, such that
\begin{multline}\label{hr3}
\|(u,v)\|_{\ZZ_T}+\sum_{k=0}^{2}\|(\partial^k_xu,\partial^k_xv)\|_{L^{\infty}_x(0,L;H^{\frac{1-k}{3}}(0,T))}\leq C\left\lbrace \|(u^0,v^0)\|_{(L^2(0,L))^2}+\|(\overrightarrow{h},\overrightarrow{g})\|_{\HH_T} \right. \\
\left.+\|(f,s)\|_{L^1(0,T;L^2(0,L))} \right\rbrace.
\end{multline}
\end{thm}

\subsection{Adjoint system}
We can now study the properties of the adjoint system of \eqref{gglin}:
\begin{equation}\label{linadjj1}
\begin{cases}
\varphi_t - \varphi_{xxx} - \frac{ab}{c}\psi_{xxx}=0,  & \text{in}\,\, (0,L)\times (0,T), \\
\psi_t  -\frac{r}{c}\psi_x-a\varphi_{xxx} -\frac{1}{c}\psi_{xxx} =0,  & \text{in}\,\, (0,L)\times (0,T), \\
\varphi(0,t)=\varphi(L,t)=\varphi_x(0,t)=0, &\text{in}\,\, (0,T),\\
\psi(0,t)=\psi(L,t)=\psi_x(0,t)=0, &\text{in}\,\, (0,T),\\
\varphi(x,0)= \varphi^0(x), \quad   \psi(x,0)= \psi^0(x),  &\text{in}\,\, (0,L).
\end{cases}
\end{equation}
Remark that the change of variable $x=L-x$ reduces system \eqref{linadjj1} to \eqref{gglin2}.
Therefore, the properties of the solutions of \eqref{linadjj1} are similar to the ones deduced
in Theorem \ref{teo1}.
\begin{prop}\label{hiddenregularities}
For any $(\varphi^0, \psi^0) \in \X$, the system \eqref{linadjj1} admits a unique solution $(\varphi, \psi) \in \ZZ_T$, such that it possess the following sharp trace properties
\begin{equation}\label{hr4}
\begin{cases}
\underset{0\leq x \leq L}{\sup} \| \partial^k_x \varphi(x,\cdot)\|_{H^{\frac{1-k}3}(0,T)}\le C_T\|\varphi^0\|_{L^2(0,L)}, \\
\underset{0\leq x \leq L}{\sup} \|  \partial^k_x \psi(x,\cdot)\|_{H^{\frac{1-k}3}(0,T)}\le C_T\|\psi^0\|_{L^2(0,L)},
\end{cases}
\end{equation}
for $k=0,1,2$, where $C_T$ increases exponentially in $T$.
\end{prop}

Applying the change of variable $t=T-t$, in what follows, we will consider the adjoint system as
\begin{equation}\label{linadj}
\begin{cases}
\varphi_t + \varphi_{xxx} + \frac{ab}{c}\psi_{xxx}=0,  & \text{in}\,\, (0,L)\times (0,T), \\
\psi_t  +\frac{r}{c}\psi_x+a\varphi_{xxx} +\frac{1}{c}\psi_{xxx} =0,  & \text{in}\,\, (0,L)\times (0,T),
\end{cases}
\end{equation}
satisfying the boundary conditions,
\begin{equation}\label{linadjbound}
\begin{cases}
\varphi(0,t)=\varphi(L,t)=\varphi_x(0,t)=0, \quad \text{in}\,\, (0,T),\\
\psi(0,t)=\psi(L,t)=\psi_x(0,t)=0,\quad \text{in}\,\, (0,T)
\end{cases}
\end{equation}
and the final conditions
\begin{equation}\label{finaladj}
\varphi(x,T)= \varphi^1(x), \qquad \psi(x,T)= \psi^1(x),  \qquad  \text{in}\,\, (0,L).
\end{equation}
Thus, the system \eqref{linadj}--\eqref{finaladj} possesses the sharp hidden regularity \eqref{hr4} a relevant result as described above.
Moreover, we have the following estimate:
\begin{prop}\label{prop3}
Any solution $(\varphi,\psi)$ of the adjoint system \eqref{linadj}--\eqref{finaladj} satisfies
\begin{align}\label{e6}
\|(\varphi^1,\psi^1)\|_{\X}^2 \leq &\frac{C}{T}\|(\varphi,\psi)\|_{L^2(0,T;\X)}^2+\frac{1}{2}\|\varphi_x(L,\cdot)\|_{L^2(0,T)}^2+\frac{b}{2c}\|\psi_x(L,\cdot)\|_{L^2(0,T)}^2 \notag\\
+&\frac{1}{2}\left\|\varphi_x(L,\cdot)+\frac{ab}{c}\psi_x(L,\cdot)\right\|_{L^2(0,T)}^2+\frac{b}{2c}\left\|a\varphi_x(L,\cdot)+\frac{1}{c}\psi_x(L,\cdot)\right\|_{L^2(0,T)}^2,
\end{align}
with $(\varphi^1,\psi^1) \in \X$ and $C=\frac{\max\{b,c\}}{\min\{b,c\}}$.
\end{prop}

\begin{proof}
Multiplying the first equation of \eqref{linadj} by $-t\varphi$, the second one by $-\frac{b}{c}t\psi$ and integrating by parts in $(0, T)\times (0,L)$, we obtain
\begin{align*}
\frac{C_1T}{2}\|(\varphi^1,\psi^1)\|_{\X}^2\leq& \frac{C_2}{2}\|(\varphi,\psi)\|^2_{L^2(0,T;\X)} - \int_0^Tt\left[\frac{b}{c}\psi(x,t)\left(a\varphi_{xx}(x,t)+\frac1c\psi_{xx}(x,t)+\frac{r}{c}\psi(x,t)\right) \right. \\
-&\left.\frac{b}{2c}\psi_x(x,t)\left(a\varphi_{x}(x,t)+\frac1c\psi_{x}(x,t)\right) -\frac{1}{2}\varphi_x(x,t)\left(\varphi_{x}(x,t)+\frac{ab}{c}\psi_{x}(x,t)\right) \right. \\
+&\left.\varphi(x,t)\left(\varphi_{xx}(x,t)+\frac{ab}{c}\psi_{xx}(x,t)\right) -\frac{br}{2c^2}\psi^2(x,t)\right]_0^Ldt,
\end{align*}
where $C_1=\min\{b,c\}$ and $C_2=\max\{b,c\}$. From \eqref{linadjbound} and applying Young inequality, \eqref{e6} is obtained.
\end{proof}

\section{Exact Boundary Controllability: Linear System}

\subsection{Four controls}
Considerations are first given to the boundary controllability of the linear system
\begin{equation}\label{ggln4}
\left\lbrace \begin{tabular}{l l}
$u_t + u_{xxx} + av_{xxx}  =0$ & in $(0,L)\times (0,T)$,\\
$v_t +\frac{r}{c}v_x+\frac{ab}{c}u_{xxx} + \frac{1}{c}v_{xxx} =0$ & in $(0,L)\times (0,T)$,\\
$u(x,0)=u^0(x), \quad v(x,0) = v^0(x)$, & in $(0,L)$
\end{tabular}\right.
\end{equation}
satisfying the boundary conditions
\begin{equation}\label{ggln4b}
\left\lbrace\begin{tabular}{l l l l}
$u(0,t) = 0$,     & $u(L,t) = 0$,      & $u_{x}(L,t) = h_2(t)$ & in $(0,T)$,\\
$v(0,t) =g_0(t)$, & $v(L,t) = g_1(t)$, & $v_{x}(L,t) = g_2(t)$ & in $(0,T)$,
\end{tabular}\right.
\end{equation}
which employ $\overrightarrow{h}_1:=(0,0,h_2)$ and $\overrightarrow{g}_1:=(g_0,g_1,g_2) \in \HH_T$.

\begin{thm}\label{teo3}
Let $L \in (0,\infty) \setminus \FF'_r$, where $\FF'_r$ is defined by \eqref{critical_diric} and $T>0$ be given. There exists a bounded linear operator
$$
\begin{array}{lcl} \Psi: &[L^2(0,L)]^2\times [L^2(0,L)]^2 \longrightarrow & \HH_T\times\HH_T
\end{array}
$$
such that for any $(u^0,v^0)\in[L^2(0,L)]^2$ and $(u^1,v^1)\in[L^2(0,L)]^2$, if one chooses $$(\overrightarrow{h}_1,\overrightarrow{g}_1)=\Psi((u^0,v^0),(u^1,v^1)),$$
then the system \eqref{ggln4}-\eqref{ggln4b} admits a solution $(u,v) \in \ZZ_T$ satisfying
\begin{equation}\label{finaldata}
u(\cdot,T)=u^1(\cdot), \qquad   \text{and}  \qquad v(\cdot,T)=v^1(\cdot).
\end{equation}
\end{thm}

To prove the previous result we first establish the following observability for the corresponding adjoint
system \eqref{linadj}-\eqref{finaladj}.
\begin{prop}\label{prop4}
For $T>0$ and $L \in (0,\infty) \setminus \FF'_r$. There exists a constant $C (T,L) > 0$, such that
\begin{align}\label{obineq1}
\|(\varphi^1, \psi^1)\|_{\X}^2\leq& C\left\lbrace  \left\|(-\Delta_t)^{-\frac16}\left(a\varphi_{xx}(L,\cdot)+\frac{1}{c}\psi_{xx}(L,\cdot)\right)\right\|_{L^2(0,T)}^2 + \left \|\varphi_x(L,\cdot)+\frac{ab}{c}\psi_x(L,\cdot)\right \|_{L^2(0,T)}^2 \right. \notag \\
+&\left. \left \|(-\Delta_t)^{-\frac16}\left(a\varphi_{xx}(0,\cdot)+\frac{1}{c}\psi_{xx}(0,\cdot)\right)\right \|_{L^2(0,T)}^2 + \left \|a\varphi_x(L,\cdot)+\frac{1}{c}\psi_x(L,\cdot)\right \|_{L^2(0,T)}^2\right\rbrace,
\end{align}
for any $(\varphi^1,\psi^1) \in \X$, where $(\varphi,\psi)$ is  solution of \eqref{linadj}-\eqref{finaladj}.
\end{prop}

\begin{proof}
We proceed as in \cite[Proposition 3.3]{rosier}. Let us suppose that \eqref{obineq1} does not hold. In this case, it follows that there exists a sequence $\{(\varphi^1_n, \psi^1_n)\}_{n \in \N}$, such that
\begin{align}
1=&\|(\varphi^1_n, \psi^1_n)\|_{\X}^2\geq n\left\lbrace \left \|(-\Delta_t)^{-\frac16}\left(a\varphi_{n,xx}(L,\cdot)+\frac{1}{c}\psi_{n,xx}(L,\cdot)\right)\right \|_{L^{2}(0,T)}^2 \right. \nonumber \\
+& \left  \|\varphi_{n,x}(L,\cdot)+\frac{ab}{c}\psi_{n,x}(L,\cdot)\right \|_{L^2(0,T)}^2 +\left \|(-\Delta_t)^{-\frac16}\left(a\varphi_{n,xx}(0,\cdot)+\frac{1}{c}\psi_{n,xx}(0,\cdot)\right)\right \|_{L^2(0,L)}^2 \label{e3}  \\
+&\left. \left \|a\varphi_{n.x}(L,\cdot)+\frac{1}{c}\psi_{n,x}(L,\cdot)\right \|_{L^2(0,T)}^2\right\rbrace. \nonumber
\end{align}
where,  for each $n\in\N$, $\{(\varphi_n,\psi_n)\}_{n\in\N}$ is the solution of  \eqref{linadj}-\eqref{finaladj}. Inequality \eqref{e3} imply that
\begin{equation}\label{e4}
\begin{cases}
(-\Delta_t)^{-\frac16}\left(a\varphi_{n,xx}(0,\cdot)+\frac{1}{c}\psi_{n,xx}(0,\cdot)\right) \rightarrow 0 & \text{in} \quad L^{2}(0,T), \\
(-\Delta_t)^{-\frac16}\left(a\varphi_{n,xx}(L,\cdot)+\frac{1}{c}\psi_{n,xx}(L,\cdot)\right) \rightarrow 0 & \text{in} \quad L^{2}(0,T), \\
\varphi_{n,x}(L,\cdot)+\frac{ab}{c}\psi_{n,x}(L,\cdot) \rightarrow 0 & \text{in} \quad L^2(0,T), \\
a\varphi_{n,x}(L,\cdot)+\frac{1}{c}\psi_{n,x}(L,\cdot)  \rightarrow 0 & \text{in} \quad L^2(0,T).
\end{cases}
\end{equation}
Since $1-a^2b >0$, from the convergence of the sequences in the third and fourth lines of \eqref{e4}, we obtain
\begin{equation}\label{e4'}
\begin{cases}
 a\varphi_{n,xx}(0,\cdot)+\frac{1}{c}\psi_{n,xx}(0,\cdot)\rightarrow 0 & \text{in} \quad  H^{-\frac13}(0,T), \\
 a\varphi_{n,xx}(L,\cdot)+\frac{1}{c}\psi_{n,xx}(L,\cdot) \rightarrow 0 & \text{in} \quad  H^{-\frac13}(0,T), \\
\varphi_{n,x}(L,\cdot) \rightarrow 0 & \text{in} \quad L^2(0,T), \\
\psi_{n,x}(L,\cdot)  \rightarrow 0 & \text{in} \quad L^2(0,T).
\end{cases}
\end{equation}
From \eqref{hr4} and \eqref{e3}, we obtain that $\{ (\varphi_n,\psi_n)\}_{n \in \N}$ is bounded in $L^2(0,T;(H^1(0,L))^2)$. On the other hand, system \eqref{linadj} implies that $\{ (\varphi_{t,n},\psi_{t,n})\}_{n \in \N}$ is bounded in $L^2(0,T; (H^{-2}(0,L))^2)$, and the compact embedding
\begin{equation}\label{e7}
 H^1(0,L) \hookrightarrow_{cc} L^2(0,L) \hookrightarrow H^{-2}(0,L),
\end{equation}
 allows us to  conclude that $\{ (\varphi_n,\psi_n)\}_{n \in \N}$ is relatively compact in $L^2(0,T;\X)$ and consequently, we obtain a subsequence, still denoted by the same index $n$, satisfying
 \begin{equation}\label{e8}
(\varphi_n,\psi_n) \rightarrow (\varphi,\psi) \mbox{ in } L^2(0,T;\X), \mbox{ as } n\rightarrow\infty.
 \end{equation}
Furthermore, \eqref{hr4} implies that $\{ \varphi_n(0,\cdot)\}_{n\in\N}$, $\{ \varphi_n(L,\cdot)\}_{n\in\N}$,  $\{\psi_n(0,\cdot)\}_{n\in\N}$ and $\{\psi_n(L,\cdot)\}_{n\in\N}$ are bounded in $H^{\frac13}(0,T)$. Then, the embedding
\begin{equation}\label{e15}
H^{\frac13}(0,T) \hookrightarrow_{cc} L^2(0,T)
\end{equation}
guarantees that the above sequences are relatively compact in $L^2(0,T)$. Thus, we obtain a subsequence, still denoted by the same index $n$, satisfying
\begin{equation}\label{e10}
\begin{cases}
\varphi_n(0,\cdot) \rightarrow \varphi(0,\cdot), \quad \varphi_n(L,\cdot) \rightarrow \varphi(L,\cdot) \qquad \text{in} \quad L^2(0,T), \\
\psi_n(0,\cdot) \rightarrow \psi(0,\cdot), \quad \psi_n(L,\cdot) \rightarrow \psi(L,\cdot) \qquad \text{in} \quad L^2(0,T).
\end{cases}
\end{equation}
From \eqref{linadjbound}, we deduce that
\begin{equation*}
\begin{cases}
\varphi(0,\cdot)=\varphi(L,\cdot)=0, \\
\psi(0,\cdot)=\psi(L,\cdot)=0.
\end{cases}
\end{equation*}
In addition, according to Proposition \ref{prop3}, we have
\begin{align*}
\|(\varphi^1_n,\psi^1_n)\|_{\X}^2 \leq &\frac{C}{T}\|(\varphi_n,\psi_n)\|_{L^2(0,T;\X)}+\frac{1}{2}\|\varphi_{n,x}(L,\cdot)\|_{L^2(0,T)}^2+\frac{b}{2c}\|\psi_{n,x}(L,\cdot)\|_{L^2(0,T)}^2 \\
+&\frac{1}{2}\left\|\varphi_{n,x}(L,\cdot)+\frac{ab}{c}\psi_{n,x}(L,\cdot)\right\|_{L^2(0,T)}^2+\frac{b}{2c}\left\|a\varphi_{n,x}(L,\cdot)+\frac{1}{c}\psi_{n,x}(L,\cdot)\right\|_{L^2(0,T)}^2.
\end{align*}
Then, from  \eqref{e4'} and \eqref{e8} if follows that  $\{(\varphi^1_n,\psi^1_n)\}_{n\in\N}$ is a Cauchy sequence in $\X$. Thus,
\begin{equation}\label{e16}
(\varphi^1_n,\psi^1_n) \rightarrow (\varphi^1,\psi^1) \mbox{ in } \X, \mbox{ as } n\rightarrow\infty.
\end{equation}
Proposition \ref{hiddenregularities} together with \eqref{e16}, imply that
\begin{equation*}
 \begin{cases}
  \varphi_{n,x}(0,\cdot) \rightarrow \varphi_x(0,\cdot)      & \mbox{ in } L^2(0,T), \mbox{ as } n\rightarrow\infty, \\
  \varphi_{n,x}(L,\cdot) \rightarrow \varphi_x(L,\cdot)      & \mbox{ in } L^2(0,T), \mbox{ as } n\rightarrow\infty, \\
 \psi_{n,x}(0,\cdot) \rightarrow \psi_x(0,\cdot)             &  \mbox{ in } L^2(0,T), \mbox{ as } n\rightarrow\infty, \\
  \psi_{n,x}(L,\cdot) \rightarrow \psi_x(L,\cdot)             &  \mbox{ in } L^2(0,T), \mbox{ as } n\rightarrow\infty \\
\end{cases}
\end{equation*}
and
\begin{equation*}
\begin{cases}
 a \varphi_{n,xx}(0,\cdot) +\frac{1}{c}\psi_{n,xx}(0,\cdot) \rightarrow a\varphi_{xx}(0,\cdot)+\frac{1}{c}\psi_{xx}(0,\cdot) & \mbox{ in } H^{-\frac13}(0,T), \mbox{ as } n\rightarrow\infty,  \\
a\varphi_{n,xx}(L,\cdot)+\frac{1}{c} \psi_{n,xx}(L,\cdot) \rightarrow a\varphi_{xx}(L,\cdot) +\frac{1}{c}\psi_{xx}(L,\cdot)        & \mbox{ in } H^{-\frac13}(0,T), \mbox{ as } n\rightarrow\infty.
 \end{cases}
 \end{equation*}
Finally, taking $n\to\infty$, from \eqref{linadj}-\eqref{finaladj} and  \eqref{e4'}, we obtain that $(\varphi,\psi)$ is solution of
\begin{equation}\label{e11}
\begin{cases}
\varphi_t + \varphi_{xxx} + \frac{ab}{c}\psi_{xxx}=0,                 & \text{in} \,\, (0,L)\times(0,T), \\
\psi_t  +\frac{r}{c}\psi_x+a\varphi_{xxx} +\frac{1}{c}\psi_{xxx} =0,  & \text{in} \,\, (0,L)\times(0,T), \\
\varphi(0,t)=\varphi(L,t)= \varphi_{x} (0,t)=0,                       &  \text{in} \,\, (0,T). \\
\psi(0,t)=\psi(L,t)= \psi_{x} (0,t)=0,                                &  \text{in} \,\, (0,T). \\
\varphi(x,T)= \varphi^1(x), \qquad \psi(x,T)= \psi^1(x),  &  \text{in} \,\, (0,L),
\end{cases}
\end{equation}
satisfying the additional boundary conditions
\begin{equation}\label{e12}
\begin{cases}
\varphi_x(L,t)=\psi_x(L,t)=0,                                         & \text{in} \,\,  (0,T),\\
 a\varphi_{xx}(0,\cdot)+\frac{1}{c}\psi_{xx}(0,\cdot)=0,  & \text{in} \,\,  (0,T), \\
 a\varphi_{xx}(L,\cdot)+\frac{1}{c}\psi_{xx}(L,\cdot)=0,  & \text{in} \,\,  (0,T),
\end{cases}
\end{equation}
and, from \eqref{e3}, we get
\begin{equation}\label{e13}
\|(\varphi^1,\psi^1)\|_{\X}=1.
\end{equation}
Notice that \eqref{e13} implies that the solutions of \eqref{e11}-\eqref{e12} can not be identically zero. However, from the following Lemma, one can conclude that $(\varphi,\psi)=(0,0)$, which drive us to contradicts \eqref{e13}.
 \end{proof}

\begin{lemma}\label{lem1}
For any $T > 0$, let $N_T$ denote the space of the initial states $(\varphi^1,\psi^1) \in \X$, such that the solution of \eqref{e11} satisfies \eqref{e12}. Then, $N_T=\{0\}$.
\end{lemma}
\begin{proof}
The proof uses the same arguments as those given in \cite{rosier}. Therefore, if $N_T\neq\{0\}$, the map  $(\varphi^1,\psi^1) \in N_T \rightarrow A(N_T)\subset \C N_T$
(where $\C N_T$ denote the complexification of $N_T$) has (at least) one eigenvalue, hence, there exists $\lambda \in \C$ and $\varphi_0,\psi_0 \in  H^3(0,L)\setminus \{ 0 \}$, such that
\begin{equation}\label{mier}
\begin{cases}
\lambda\varphi_0+ \varphi'''_0 + \frac{ab}{c}\psi'''_0=0, & \text{in} \,\, (0,L),  \\
\lambda\psi_0 +\frac{r}{c}\psi'_0+a\varphi'''_{0} +\frac{1}{c}\psi'''_{0}=0, & \text{in} \,\, (0,L),  \\
\varphi_0(0)=\varphi_0(L)=\varphi'_0(0)=\varphi'_0(L)=0, \\
\psi_0(0)=\psi_0(L)=\psi'_0(0) =\psi'_0(L)=0, \\
a\varphi''_0(0)+\frac{1}{c}\psi''_0(0)=0,\\
a\varphi''_0(L)+\frac{1}{c}\psi''_0(L)=0.
\end{cases}
\end{equation}
To conclude the proof of the Lemma, we prove that this does not hold if $L \in (0,\infty) \setminus \FF'_r$.
\end{proof}
To simplify the notation, henceforth we denote $(\varphi_0,\psi_0):=(\varphi,\psi)$. Moreover, the notation $\{0,L\}$ means that the function is applied to $0$ and $L$, respectively.
\begin{lemma}\label{lem2}
Let $L>0$ and consider the assertion
\begin{equation*}
(\NN):\ \ \exists \lambda \in \C,  \exists (\varphi,\psi) \in (H^3(0,L))^2 \setminus (0,0)\,\, \text{such that}\,\,
\begin{cases}
\lambda\varphi+ \varphi''' + \frac{ab}{c}\psi'''=0, & \text{in} \,\, (0,L),  \\
\lambda\psi +\frac{r}{c}\psi'+a\varphi''' +\frac{1}{c}\psi'''=0, & \text{in} \,\, (0,L),  \\
\varphi(x)=\psi(x)=0, & \text{in} \,\, \{0,L\},\\
\varphi'(x)=\psi'(x)=0, & \text{in} \,\, \{0,L\},\\
a\varphi'' (x) +\frac{1}{c}\psi''(x) =0, & \text{in} \,\, \{0,L\}.
\end{cases}
\end{equation*}
Then, $(\NN)$ holds if and only if $L \in \FF'_r$.
\end{lemma}
\begin{proof}
We use an argument which is similar to the one used in \cite[Lemma 3,5]{rosier}. Let us introduce the notation $\hat{\varphi}(\xi) =\int_0^L e^{-i x \xi}\varphi(x)dx$ and $\hat{\psi}(\xi) =\int_0^L e^{-i x \xi}\psi(x)dx$. Then, multiplying the equations by $e^{-i x \xi}$, integrating by parts over $(0,L)$ and using the boundary conditions, we have
\begin{equation}\label{eqq1}
\begin{cases}
[(i\xi)^3+\lambda]\hat{\varphi}(\xi)+\dfrac{ab}{c}(i\xi)^3\hat{\psi}(\xi) =\varphi''(0)+\dfrac{ab}{c}\psi''(0)- \left(\varphi''(L)+\dfrac{ab}{c}\psi''(L)\right)e^{-iL\xi}, \\
\dfrac{1}{c}[(i\xi)^3+r(i\xi)+c\lambda]\hat{\psi}(\xi) + a(i\xi)^3\hat{\varphi}(\xi)=0.
\end{cases}
\end{equation}
From the first equation in \eqref{eqq1}, we have
\begin{equation}\label{eqq2}
\hat{\varphi}(\xi) = \frac{\left(\alpha+\beta e^{-iL\xi}\right)}{(i\xi)^3+\lambda} -\frac{ab(i\xi)^3\hat{\psi}(\xi)}{c\left((i\xi)^3+\lambda\right)},
\end{equation}
where $\alpha= \varphi''(0)+\frac{ab}{c}\psi''(0)$ and $\beta=-\varphi''(L)-\frac{ab}{c}\psi''(L)$. Replacing \eqref{eqq2} in the second equation of \eqref{eqq1}, it follows that $$ \frac{1}{c}\left[ (i\xi)^3+r(i\xi)+c\lambda - \frac{a^2b(i\xi)^6}{(i\xi)^3+\lambda}\right]\hat{\psi}(\xi) =-\frac{a(i\xi)^3\left(\alpha+\beta e^{-iL\xi}\right)}{(i\xi)^3+\lambda}.$$
Therefore,
\begin{equation}\label{eqq3}
\hat{\psi}(\xi) =-\frac{ac(i\xi)^3\left(\alpha+\beta e^{-iL\xi}\right)}{(1-a^2b)(i\xi)^6+r(i\xi)^4+(c+1)\lambda(i\xi)^3+r\lambda(i\xi)+c\lambda^2}.
\end{equation}
Having \eqref{eqq3} in hands, from \eqref{eqq2} we obtain
\begin{align*}
\hat{\varphi}(\xi)& = \left( 1 + \frac{a^2b(i\xi)^6}{(1-a^2b)(i\xi)^6+r(i\xi)^4+(c+1)\lambda(i\xi)^3+r\lambda(i\xi)+c\lambda^2}\right)\frac{\left(\alpha+\beta e^{-iL\xi}\right)}{(i\xi)^3+\lambda},
\end{align*}
hence,
\begin{align*}
\hat{\varphi}(\xi)& =  \frac{\left( (i\xi)^3+r(i\xi)+c\lambda \right)\left(\alpha+\beta e^{-iL\xi}\right)}{(1-a^2b)(i\xi)^6+r(i\xi)^4+(c+1)\lambda(i\xi)^3+r\lambda(i\xi)+c\lambda^2}.
\end{align*}
Setting $\lambda=ip$, $p \in \C$, we have that  $\hat{\psi}(\xi)=-acif(\xi)$ and $\hat{\varphi}(\xi)=ig(\xi)$, where
\begin{equation*}
\begin{cases}
f(\xi) = \dfrac{\xi^3\left( \alpha +\beta e^{-iL\xi}\right)}{P(\xi)}, \\
g(\xi)= \dfrac{\left(\xi^3-r\xi-cp\right)\left( \alpha +\beta e^{-iL\xi}\right)}{P(\xi)}.
\end{cases}
\end{equation*}
with
\begin{equation*}
P(\xi):=(1-a^2b)\xi^6-r\xi^4-(c+1)p\xi^3+rp\xi+cp^2.
\end{equation*}
Using Paley-Wiener theorem (\cite[Section 4, page 161]{yosida}) and the usual characterization of $H^2(\R)$ functions by means of their Fourier transforms, we see that $(\NN)$ is equivalent to the existence of $p \in \C$ and $(\alpha,\beta) \in \C^2 \setminus (0,0),$ such that
\begin{enumerate}
\item[(i)] $f$ and $g$ are entire functions in $\C$,
\item[(ii)] $\displaystyle\int_{\R}|f(\xi)|^2(1+|\xi|^2)^2d\xi<\infty$ and $\int_{\R}|g(\xi)|^2(1+|\xi|^2)^2d\xi<\infty$,
\item[(iii)] $\forall \xi \in \C$, we have that $|f(\xi)|\leq c_1(1+|\xi|)^ke^{L|Im\xi|}$ and  $|g(\xi)|\leq c_1(1+|\xi|)^ke^{L|Im\xi|}$, for some positive constants $c_1 $ and $k$.
\end{enumerate}
Notice that if (i) holds true, then (ii) and (iii) are satisfied. Recall that $f$ and $g$ are entire functions if only if, the roots $\xi_0, \xi_1, \xi_2, \xi_3, \xi_4$ and $\xi_5$ of $P(\xi)$ are roots of $\xi^3\left( \alpha +\beta e^{-iL\xi}\right)$ and $(\xi^3-r\xi-cp)\left( \alpha +\beta e^{-iL\xi}\right)$.
\vglue 0.2 cm
Let us first assume that $\xi=0$ is not root of $P(\xi)$.  Thus, it is sufficiently to consider the  case when $\alpha +\beta e^{-iL\xi}$ and $P(\xi)$ share the same roots. Since the roots of $\alpha +\beta e^{-iL\xi}$ are simple, unless $\alpha = \beta= 0$ (Indeed, it implies that $\varphi''(0)+\frac{ab}{c}\psi''(0)=0$ and $\varphi''(L)+\frac{ab}{c}\psi''(L)=0$, thus, using the system \eqref{mier}, we conclude that $(\varphi,\psi)=(0,0)$, which is a contradiction). Then, (i) holds provided that the roots of $P(\xi)$ are simple. Thus, we conclude tha $(\NN)$ is equivalent to the existence of complex numbers $p$, $\xi_0$ and positive integers $k,l,m,n$ and $s$, such that, if we set
\begin{equation}\label{eqq4}
\xi_1=\xi_0+\frac{2\pi}{L}k, \quad \xi_2=\xi_1+\frac{2\pi}{L}l, \quad \xi_3=\xi_2+\frac{2\pi}{L}m, \quad \xi_4=\xi_3+\frac{2\pi}{L}n\quad \text{and} \quad \xi_5=\xi_4+\frac{2\pi}{L}s,
\end{equation}
we have
\begin{equation}\label{eqq5}
P(\xi)=(\xi-\xi_0)(\xi-\xi_1)(\xi-\xi_2)(\xi-\xi_3)(\xi-\xi_4)(\xi-\xi_5).
\end{equation}
In particular, we obtain the following relations:
\begin{gather}
\xi_0+\xi_1+\xi_2+\xi_3+\xi_4+\xi_5=0, \label{eqq6}
\end{gather}
\begin{multline}
\xi_0(\xi_1+\xi_2+\xi_3+\xi_4+\xi_5)+\xi_1(\xi_2+\xi_3+\xi_4+\xi_5)+\xi_2(\xi_3+\xi_4+\xi_5) \\
+\xi_3(\xi_4+\xi_5) +\xi_4\xi_5=-\frac{r}{1-a^2b}, \label{eqq7}
\end{multline}
\begin{gather}
\xi_0\xi_1\xi_2\xi_3\xi_4\xi_5=\left(\frac{c}{1-a^2b}\right) p^2. \label{eqq8}
\end{gather}
Some calculations lead to
\begin{equation}\label{eqq14}
\begin{cases}
L=\pi \sqrt{\dfrac{(1-a^2b)\alpha(k,l,m,n,s)}{3r}},\\
\\
\xi_0 = -\dfrac{\pi}{3}(5k+4l+3m+2n+s),\\
\\
p= \sqrt{\dfrac{(1-a^2b)\xi_0\xi_1\xi_2\xi_3\xi_4\xi_5}{c}},
\end{cases}
\end{equation}
where
\begin{multline*}
\alpha(k,l,m,n,s):=5k^2+8l^2+9m^2+8n^2+5s^2+8kl+6km+4kn+2ks+12ml \\
+8ln+3ls+12mn+6ms+8ns.
\end{multline*}
Finally, we assume that $\xi_0= 0$ is a root of $P(\xi)$. In this case, it follows that $p=0$ and, therefore,
\begin{equation*}
\begin{cases}
f(\xi)= \dfrac{\xi^3\left(\alpha+\beta e^{-iL\xi}\right)}{(1-a^2b)\xi^6-r\xi^4}=\dfrac{\left(\alpha+\beta e^{-iL\xi}\right)}{\xi\left((1-a^2b)\xi^2-r\right)},\\
\\
g(\xi)= \dfrac{\left(\xi^3-r\xi\right)\left( \alpha +\beta e^{-iL\xi}\right)}{(1-a^2b)\xi^6-r\xi^4}=\dfrac{\left(\xi^2-r\right)\left( \alpha +\beta e^{-iL\xi}\right)}{\xi^3\left((1-a^2b)\xi^2-r\right)}.
\end{cases}
\end{equation*}
Then, $(\NN)$ holds if and only if  $f$ and $g$ satisfy (i), (ii) and (iii). Thus (i) holds provided that
\begin{equation*}
\xi_0=0, \quad \xi_1=\sqrt{\frac{r}{1-a^2b}} \quad \text{and} \quad \xi_2=-\sqrt{\frac{r}{1-a^2b}}
\end{equation*}
are roots of $\alpha+\beta e^{-iL\xi}$. Note that, zero must be root of multiplicity three, which leads to a contradiction. Thus, $\xi=0$ is not root of $P(\xi)$. Finally, from \eqref{eqq14}, we deduce that $(\NN)$ holds if and only if $L \in \FF'_r$. This completes the proof of Lemma \ref{lem2}, Lemma \ref{lem1} and, consequently, the proof of Proposition \ref{prop4}.
\end{proof}

\begin{proof}[\textbf{Proof of Theorem \ref{teo3}}]
Without loss of generality, we assume that $(u^0,v^0)=(0,0)$. Let $(\varphi,\psi)$ be a solution of \eqref{linadj}-\eqref{finaladj}. Multiplying the equations in \eqref{ggln4} by the solution $(\varphi,\psi)$ and integrating by parts we have
\begin{align*}
 \int_0^L\left( u(x,T)\varphi(x,T)+v(x,T)\psi(x,T)\right)dx =&  \int_0^T g_0(t)\left(a\varphi_{xx}(0,t)+\frac{1}{c}\psi_{xx}(0,t)\right)dt \nonumber \\
-&\int_0^T g_1(t)\left(a\varphi_{xx}(L,t)+\frac{1}{c}\psi_{xx}(L,t)\right)\label{ad2} \\
+&\int_0^T g_2(t)\left(a\varphi_x(L,t)+\frac{1}{c}\psi_x(L,t)\right)dt \nonumber \\
+&\int_0^T h_2(t)\left(\varphi_x(L,t)+\frac{ab}{c}\psi_x(L,t)\right)dt \nonumber
\end{align*}

Let us denote by $\Gamma$ the linear and bounded map  defined by
\begin{equation*}
\begin{tabular}{r c c c}
$\Gamma :$ & $L^2(0, L) \times L^2(0, L)$        &  $\longrightarrow$ & $L^2(0, L) \times L^2(0, L)$ \\
           & $(\varphi^1(\cdot), \psi^1(\cdot))$ &  $\longmapsto$     & $\Gamma(\varphi^1(\cdot), \psi^1(\cdot))=(u(\cdot,T), v(\cdot,T))$
\end{tabular}
\end{equation*}
where $(u,v)$ is the solution of \eqref{ggln4}-\eqref{ggln4b},  with
\begin{equation*}
\begin{cases}
g_0(t) = (-\Delta_t)^{-\frac13}\left(a\varphi_{xx}(0,t)+\frac{1}{c}\psi_{xx}(0,t)\right), & g_2(t)= a\varphi_x(L,t)+\frac{1}{c}\psi_x(L,t), \\
g_1(t)=-(-\Delta_t)^{-\frac13}\left(a\varphi_{xx}(L,t)+\frac{1}{c}\psi_{xx}(L,t)\right),   & h_2(t)=\varphi_x(L,t)+\frac{ab}{c}\psi_x(L,t),
\end{cases}
\end{equation*}
where $(\varphi,\psi)$ the solution of the system \eqref{linadj}-\eqref{finaladj} with initial data $(\varphi^1,\psi^1)$. According to Proposition \ref{prop4}
\begin{align*}
\left (\Gamma(\varphi^1, \psi^1), (\varphi^1, \psi^1) \right)_{(L^2(0,L))^2} \geq C^{-1} \|(\varphi^1, \psi^1)\|_{\X}^2.
\end{align*}
The proof is complete by using the Lax-Milgram Theorem.
\end{proof}

\subsection{One control}
Consider the boundary controllability of the linear system employing only one control input $h_2$ and fixing $h_0=h_1=g_0=g_1=0$, namely,
\begin{equation}\label{gglnb3}
\left\lbrace\begin{tabular}{l l l l}
$u(0,t) = 0$ & $u(L,t) = 0,$ & $u_{x}(L,t) = h_2(t)$,& in $(0,T)$,\\
$v(0,t) =0,$ & $v(L,t) = 0,$ & $v_{x}(L,t) =0$,& in $(0,T)$.
\end{tabular}\right.
\end{equation}
Note that by using the change of variable $x' = L-x$ and $t' = T-t$, the system \eqref{linadj}-\eqref{finaladj} is equivalent to the following forward system
\begin{align}
&\begin{cases}\label{linadj1}
\varphi_t +\varphi_{xxx} +\frac{ab}{c}\psi_{xxx}=0,  & \text{in} \,\, (0,L)\times (0,T),\\
\psi_t  +\frac{r}{c}\psi_x+a\varphi_{xxx} +\frac{1}{c}\psi_{xxx} =0,  & \text{in} \,\, (0,L)\times (0,T), \\
\varphi(x,0)= \varphi^0(x), \,\,
\psi(x,0)= \psi^0(x), & \text{in} \,\, (0,L),
\end{cases}
\end{align}
with boundary conditions
\begin{equation}\label{linadjbound1}
\begin{cases}
\varphi(0,t)= \varphi(L,t)=\varphi_x(L,t)=0, & \text{in} \,\, (0,T),\\
\psi(0,t)=\psi(L,t)=\psi_x(L,t)=0, & \text{in} \,\,  (0,T). \\
\end{cases}
\end{equation}

In this case, the observability inequality
\begin{align}\label{obineq3}
\|(\varphi^0, \psi^0)\|_{\X}^2\leq C  \left  \|\varphi_x(0,\cdot)+\frac{ab}{c}\psi_x(0,\cdot)\right \|_{L^2(0,T)}^2
\end{align}
plays a fundamental role for the study of the controllability. To prove \eqref{obineq3} we use a direct approach based on the multiplier technique and the estimates given by the hidden regularity. Such estimates give us the observability inequality for some values of the length $L$ and time of control $T$.
\begin{prop}\label{prop6}
Let us suppose that  $T > 0$ and $L>0$ satisfy
\begin{align}\label{Lsmall}
L< \frac{\min\{b,c\}}{  \max\{b,c\} \beta C_T}T,
\end{align}
where $C_T$ is the constant in \eqref{hr4} and $\beta$ is the constant given by the embedding $H^{\frac{1}{3}}(0,T) \subset L ^2(0,T)$.
Then, there exists a constant $C (T,L) > 0$, such that for any $(\varphi^0,\psi^0)$ in $\X$ the observability inequality \eqref{obineq3} holds, for any $(\varphi,\psi)$ solution of \eqref{linadj1}-\eqref{linadjbound1} with initial data $(\varphi^0,\psi^0)$.
\end{prop}

\begin{proof}
We multiply the first equation in \eqref{linadj1} by $(T-t)\varphi$, the second one by $\frac{b}{c}(T-t)\psi$ and integrate over $(0,T)\times(0,L)$. Thus, we obtain

\begin{align*}
\frac{T}{2}\int_0^L\left( \varphi_0^2(x) +\frac{b}{c}\psi_0^2(x)\right)dx=&\frac{1}{2}\int_0^T\int_0^L \left(\varphi^2(x,t)+\frac{b}{c}\psi^2(x,t)\right)dxdt \\
+& \frac12\int_0^T(T-t)\left[\varphi_x^2(0,t)+\frac{2ab}{c}\psi_x(0,t)\varphi_x(0,t)+\frac{b}{c^2}\psi_x^2(0,t)\right]dt.
\end{align*}
Consequently,
\begin{align}\label{TLcondition1}
\|(\varphi^0,\psi^0)\|_{\X}^2\leq \frac{C}{T}\|(\varphi,\psi)\|_{L^2(0,T;\X)}^2 +C_1\left\|\varphi_x(0,\cdot)+\frac{ab}{c}\psi_x(0,\cdot)\right\|_{L^2(0,T)}^2,
\end{align}
where $C=\frac{\max\{b,c\}}{\min\{b,c\}}$ and $C_1=C_1(a,b,c)>0$.  On the other hand, note that
\begin{equation*}
\|\varphi(\cdot,t)\|^2_{L^2(0,L)} \leq L \|\varphi(\cdot,t)\|^2_{L^{\infty}(0,L)}, \quad \text{and} \quad \|\psi(\cdot,t)\|^2_{L^2(0,L)} \leq L \|\psi(\cdot,t)\|^2_{L^{\infty}(0,L)}.
\end{equation*}
Hence,
\begin{align}\label{TLcondition2}
\|(\varphi,\psi)\|_{L^2(0,T;\X)}^2 &\leq L\int_0^T \left\lbrace \frac{b}{c} \|\varphi(\cdot,t)\|^2_{L^{\infty}(0,L)}  + \|\psi(\cdot,t)\|^2_{L^{\infty}(0,L)}\right\rbrace dt \\
&\leq \frac{bL\beta}{c} \|\varphi\|_{H^{\frac{1}{3}}(0,T;L^{\infty}(0,L))}^2+   L\beta\|\psi\|_{H^{\frac{1}{3}}(0,T;L^{\infty}(0,L))}^2
\end{align}
where $\beta$ is the constant given by the compact embedding $H^{\frac{1}{3}}(0,T) \subset L ^2(0,T)$. Combining \eqref{TLcondition1}, \eqref{TLcondition2} and Proposition \ref{hiddenregularities}, we obtain
\begin{align*}
\|(\varphi^0,\psi^0)\|_{\X}^2 &\leq  \frac{L \beta C_TC}{T} \|(\varphi^0,\psi^0)\|_{\X}^2   +C_1\left\|\varphi_x(0,\cdot) +\frac{ab}{c}\psi_x(0,\cdot)\right\|_{L^2(0,T)}^2.
\end{align*}
Finally, we obtain
\begin{align*}
\|(\varphi^0,\psi^0)\|_{\X}^2 \leq & K \left\|\varphi_x(0,\cdot)+\frac{ab}{c}\psi_x(0,\cdot)\right\|_{L^2(0,T)}^2
\end{align*}
under the condition
\begin{equation}\label{TLcondition}
 K= C_1\left(1- \frac{ CC_T \beta L}{T}\right)^{-1}>0.
\end{equation}
\end{proof}
From the observability inequality \eqref{obineq3}, the following result holds.
\begin{thm}\label{teo6}
Let $T > 0$ and $L >0$ satisfying \eqref{Lsmall}. Then, the system \eqref{ggln4}-\eqref{gglnb3} is exactly controllable in time T.
\end{thm}
\begin{proof}
Consider the map
\begin{equation*}
\begin{tabular}{r c c c}
$\Gamma :$ & $L^2(0, L) \times L^2(0, L)$        &  $\longrightarrow$ & $L^2(0, L) \times L^2(0, L)$ \\
           & $(\varphi^1(\cdot), \psi^1(\cdot))$ &  $\longmapsto$     & $\Gamma(\varphi^1(\cdot), \psi^1(\cdot))=(u(\cdot,T), v(\cdot,T))$
\end{tabular}
\end{equation*}
where $(u,v)$ is the solution of \eqref{ggln4}-\eqref{gglnb3},  with $h_2(t)=\varphi_x(L,t)+\frac{ab}{c}\psi_x(L,t)$ and $(\varphi,\psi)$ is the solution of the system \eqref{linadj}-\eqref{finaladj} with initial data $(\varphi^1,\psi^1)$. From \eqref{obineq3} and the Lax-Milgram theorem, the proof is achieved.
\end{proof}

\section{Exact Controllability: The Nonlinear Control System}
\subsection{Well-posedness of the nonlinear system}
In this subsection, attention will be given to the full nonlinear initial boundary value problem (IBVP)
\begin{equation}
\label{gg1_new}
\begin{cases}
u_t + uu_x+u_{xxx} + a v_{xxx} + a_1vv_x+a_2 (uv)_x =0, & \text{in} \,\, (0,L)\times (0,T),\\
c v_t +rv_x +vv_x+abu_{xxx} +v_{xxx}+a_2buu_x+a_1b(uv)_x  =0,  & \text{in} \,\, (0,L)\times (0,T), \\
u(x,0)= u^0(x), \quad v(x,0)=  v^0(x), & \text{in} \,\, (0,L),
\end{cases}
\end{equation}
with the boundary conditions
\begin{equation}\label{gg2_new}
\begin{cases}
u(0,t)=h_0(t),\,\,u(L,t)=h_1(t),\,\,u_{x}(L,t)=h_2(t),\\
v(0,t)=g_0(t),\,\,v(L,t)=g_1(t),\,\,v_{x}(L,t)=g_2(t).
\end{cases}
\end{equation}
We show that the IBVP \eqref{gg1_new}-\eqref{gg2_new} is locally well-posed in the space $\ZZ_T$.

\begin{thm}\label{nonlinearteo}
Let $T>0$ be given. For any $(u^0,v^0) \in \X$ and $\overrightarrow{h}:=(h_0,h_1,h_2)$, $\overrightarrow{g}:=(g_0,g_1,g_2) \in \HH_T$, there exists $T^{*}\in (0, T]$ depending on $\|(u^0,v^0)\|_{\X}$, such that the IBVP  \eqref{gg1_new}-\eqref{gg2_new} admits a unique solution $
(u,v) \in \ZZ_{T^*}$
with
\begin{equation*}
\partial_x^k u,\partial_x^k v \in L^{\infty}_x(0,L;H^{\frac{1-k}{3}}(0,T^*)),  \quad k=0,1,2.
\end{equation*}
Moreover, the corresponding solution map is Lipschitz continuous.
\end{thm}

\begin{proof}
Let
$$\FF_T= \left\lbrace (u,v) \in \ZZ_T: (u,v) \in  L^{\infty}_x(0,L;(H^{\frac{1-k}{3}}(0,T))^2), k=0,1,2\right\rbrace$$
be a Banach space equipped with the norm
\[
\|(u,v)\|_{\FF_T} = \|(u,v)\|_{\ZZ_T}+\sum_{k=0}^{2}\|(\partial_x^ku,\partial_x^kv)\|_{L^{\infty}_x(0,L;(H^{\frac{1-k}{3}}(0,T))^2)}.
\]
Let $0< T^* \leq T$ to be determined later. For each $u,v \in \FF_{T^*}$, consider the problem
\begin{equation}\label{e1'}
\left\lbrace \begin{tabular}{l l}
$\omega_t + \omega_{xxx} + a\eta_{xxx}  =f(u,v)$, & in $(0,L)\times (0,T^*)$,\\
$\eta_t +\frac{ab}{c}\omega_{xxx} + \frac{1}{c}\eta_{xxx} =s(u,v)$, & in $(0,L)\times (0,T^*)$,\\
$\omega(0,t) = h_0(t),\,\,\omega(L,t) = h_1(t),\,\,\omega_{x}(L,t) = h_2(t)$,& in $(0,T^*)$,\\
$\eta(0,t) =g_0(t),\,\,\eta(L,t) = g_1(t),\,\, \eta_{x}(L,t) = g_2(t)$,& in $(0,T^*)$,\\
$\omega(x,0)=u^0(x), \quad v(x,0) = v^0(x)$, & in $(0,L)$,
\end{tabular}\right.
\end{equation}
where $$f(u,v)=-a_1(vv_x)-a_2(uv)_x$$ and $$s(u,v)=-\frac{r}{c}v_x -\frac{a_2b}{c}(uu_x)-\frac{a_1b}{c}(uv)_x.$$
Since $\|v_x\|_{L^1(0,\beta;L^2(0,L))}\leq \beta^{\frac12}\|v\|_{\ZZ_{\beta}}$, from \cite[Lemma 3.1]{bonasunzhang2003} we deduce that $f(u,v)$ and $s(u,v)$ belong to $L^1(0,T^*;L^2(0,L))$
and
\begin{equation*}
\|(f,s)\|_{L^1(0,T^*;(L^2(0,L))^2)}\leq C_1 ((T^*)^{\frac12}+(T^*)^{\frac13})\left( \|u\|_{\ZZ_{T^*}}^2+(\|u\|_{\ZZ_{T^*}}+1)\|v\|_{\ZZ_{T^*}}+\|v\|_{\ZZ_{T^*}}^2\right),
\end{equation*}
for some positive constant $C_1$. According to Proposition \ref{prop1}, we can define the operator
\[
\Gamma: \FF_{T^*} \rightarrow \FF_{T^*} \quad \text{given by} \quad \Gamma(u,v)=(\omega,\eta),
\]
where $(\omega,\eta)$ is the solution of (\ref{e1'}). Moreover,
\begin{equation*}
\|\Gamma(u,v)\|_{\FF_{T^*}}\leq C\left\lbrace \|(u^0,v^0)\|_{\X}+\|(\overrightarrow{h},\overrightarrow{g})\|_{\HH_T^*}+\|(f,s)\|_{L^1(0,T^*;(L^2(0,L))^2)} \right\rbrace,
\end{equation*}
where the positive  constant $C$ depends only on $T^*$. Thus, we obtain
\begin{align*}
\|\Gamma(u,v)\|_{\FF_{T^*}}\leq &C\left\lbrace \|(u^0,v^0)\|_{\X}+\|(\overrightarrow{h},\overrightarrow{g})\|_{\HH_T^*}\right\rbrace \\
+&CC_1((T^*)^{\frac12}+(T^*)^{\frac13})\left( \|u\|_{\ZZ_{T^*}}^2+(\|u\|_{\ZZ_{T^*}}+1)\|v\|_{\ZZ_{T^*}}+\|v\|_{\ZZ_{\beta}}^2\right).
\end{align*}
Let $(u,v) \in B_{r}(0)$, where
$$B_{r}(0):=\left\lbrace (u,v) \in \FF_{T^*}: \|(u,v)\|_{\FF_{T^*}}\leq r\right\rbrace,$$
with $r=2C\left\lbrace \|(u^0,v^0)\|_{\X}+\|(\overrightarrow{h},\overrightarrow{g})\|_{\HH_T}\right\rbrace$. It follows that
\begin{equation}\label{e2'}
\|\Gamma(u,v)\|_{\FF_{T^*}}\leq \frac{r}{2}+CC_1((T^*)^{\frac12}+(T^*)^{\frac13})\left( 3r+1\right)r.
\end{equation}
Choosing $T^*>0$, such that
$$CC_1((T^*)^{\frac12}+(T^*)^{\frac13})\left( 3r+1\right)\leq \frac{1}{2},$$
from (\ref{e2'}), we have
$$\|\Gamma(u,v)\|_{\FF_{T^*}}\leq r.$$
Therefore,
\[
\Gamma: B_r(0)\subset\FF_{T^*} \rightarrow B_r(0).
\]
On the other hand,  $\Gamma(u_1,v_1)-\Gamma(u_2,v_2)$ is the solution of system
\begin{equation*}
\left\lbrace \begin{tabular}{l l}
$\omega_t + \omega_{xxx} + a\eta_{xxx}  =f(u_1,v_1)-f(u_2,v_2)$, & in $(0,L)\times (0,T^*)$,\\
$\eta_t +\frac{ab}{c}\omega_{xxx} + \frac{1}{c}\eta_{xxx} =s(u_1,v_1)-s(u_2,v_2) $, & in $(0,L)\times (0,T^*)$,\\
$\omega(0,t) = \omega(L,t) = \omega_{x}(L,t) = 0$,& in $(0,T^*)$,\\
$\eta(0,t)=\eta(L,t) = \eta_{x}(L,t) = 0$,& in $(0,T^*)$,\\
$\omega(x,0)=0, \quad v(x,0) = 0$, & in $(0,L)$.
\end{tabular}\right.
\end{equation*}
Note that
\begin{align*}
|f(u_1,v_1)-f(u_2,v_2)| \leq C_2|\left( (v_2-v_1)v_{2,x}+ v_1(v_2-v_1)_x+(u_2(v_2-v_1))_x +((u_2-u_1)v_1)_x\right)|
\end{align*}
and
\begin{align*}
|s(u_1,v_1)-s(u_2,v_2)| \leq& C_2|\left( (v_2-v_1)_x+ (u_2-u_1)u_{2,x}+ u_1(u_2-u_1)_x \right.\\
+&\left.(u_2(v_2-v_1))_x +((u_2-u_1)v_1)_x\right)|,
\end{align*}
for some positive constant $C_2$. Proposition \ref{prop1} and \cite[Lemma 3.1]{bonasunzhang2003} give us the following estimate
\begin{align*}
\|\Gamma(u_1,v_1)-\Gamma(u_2,v_2)\|_{\FF_{T^*}}
&\leq C_3((T^*)^{\frac{1}{2}}+(T^*)^{\frac{1}{3}})(8r+1)\|(u_1-u_2,v_1-v_2)\|_{\FF_{T^*}},
\end{align*}
for some positive constant $C_3$. Choosing $T^*$, such that
$$C_3((T^*)^{\frac{1}{2}}+(T^*)^{\frac{1}{3}})(8r+1) \leq \frac{1}{2},$$
we obtain
\[\|\Gamma(u_1,v_1)-\Gamma(u_2,v_2)\|_{\FF_{T^*}}\leq \frac{1}{2}\|(u_1-u_2,v_1-v_2)\|_{\FF_{T^*}}.
\]
Hence $\Gamma: B_r(0) \rightarrow B_r(0)$ is a contraction and, by Banach fixed point theorem, we obtain a unique $(u,v) \in B_r(0)$, such that $\Gamma(u,v) = (u,v) \in \FF_{T^*}$ and, therefore, the proof is complete.
\end{proof}

We are now in position to prove our main result. First,  define the bounded linear operators
\begin{equation}\label{operatorcontro}
\Lambda_i : \X \times \X \longrightarrow \HH_T \times \HH_T \qquad (i=1,2),
\end{equation}
such that, for any $(u^0, v^0) \in \X$ and $(u^1, v^1) \in \X$,
$$\Lambda_i\left(  \left( \begin{array}{cc} u^0\\v^0 \end{array}\right), \left( \begin{array}{cc} u^1\\v ^1 \end{array}\right)\right  ):= \left( \begin{array}{cc} \vec{h}_i\\ \vec{g}_i \end{array}\right), $$
where
\begin{enumerate}
\item[(i)] $\vec{h}_1=(0,0,h_2)$ and $\vec{g}_1=(g_0,g_1,g_2)
$,
\item[(ii)] $\vec{h}_2=(0,0,h_2)$ and $\vec{g}_2=(0,0,0)$.
\end{enumerate}
\vglue 0.2 cm
\begin{proof}[\textbf{Proof of Theorem \ref{main}}]
 According to Proposition \ref{prop1} and \cite[Theorem 2.10]{bonasunzhang2003} the solution of  \eqref{gg1_new}-\eqref{gg2_new}  can be written as:
\begin{align*}
\left( \begin{array}{cc} u(t)\\v(t) \end{array}\right) =&  W_0(t) \left( \begin{array}{cc} u_0\\v_0 \end{array}\right) + W_{bdr}(t)\left(\begin{array}{cc}\vec{h}_i \\  \vec{g}_i   \end{array}\right)\\
-& \int_0^t W_0 (t-\tau)\left( \begin{array}{cc}  a_1 (vv_x)(\tau)+a_2 (uv)_x(\tau) \\  \frac{r}{c}v_x(\tau)+ \frac{a_2b}{c}(uu_x)(\tau)+\frac{a_1b}{c}(uv)_x(\tau) \end{array}\right) d\tau,
\end{align*}
with $i=1,2$, where $\{W_0(t)\}_{t\ge 0}$ and  $\{W_{bdr}(t)\}_{t\ge 0}$ are  the operators defined in the proof of Proposition \ref{prop1}.
\vglue 0.2 cm
For $u,v \in \ZZ_T$,  let us define
$$\left(\begin{array}{cc} \upsilon \\ \nu (T,u,v)\end{array}\right)  := \int_0^T W_0(T-\tau)   \left( \begin{array}{cc}  a_1  (vv_x)(\tau)  + a_2 (uv)_x (\tau) \\
 \frac{a_2b}{c}(uu_x)(\tau) + \frac{a_2b}{c} (uv)_x(\tau) \end{array} \right)  d\tau.  $$
Here, we consider the case $i=1$. The other case $i=2$ is analogous and, therefore, we will omit it. Consider the map
\begin{align*}
\Gamma\left( \begin{array}{cc} u\\v \end{array}\right) = & W_0(t) \left( \begin{array}{cc} u^0\\v^0 \end{array}\right) + W_{bdr}(x)  \Lambda_1\left(    \left( \begin{array}{cc}   u^0\\v^0  \end{array}\right),  \left( \begin{array}{cc}  u^1 \\ v^1 \end{array}\right) +  \left( \begin{array}{cc} v \\  \nu(T,u,v) \end{array}\right) \right) \\
-& \int_0^t W_0 (t-\tau)
\left( \begin{array}{cc}  a_1 (vv_x)(\tau)+a_2 (uv)_x(\tau) \\  \frac{r}{c}v_x(\tau)+ \frac{a_2b}{c}(uu_x)(\tau)+\frac{a_1b}{c}(uv)_x(\tau) \end{array}\right) d\tau.
\end{align*}
By choosing
\begin{equation}\label{controli}
\left( \begin{array}{cc} \vec{h}_1 \\ \vec{g}_1 \end{array}\right)  = \Lambda_1\left(    \left( \begin{array}{cc}   u^0\\v^0  \end{array}\right),  \left( \begin{array}{cc}  u^1 \\ v^1 \end{array}\right) +  \left( \begin{array}{cc} v \\  \nu(T,u,v) \end{array}\right) \right),
\end{equation}
we get, from Theorem \ref{teo3},
$$\Gamma\left( \begin{array}{cc} u\\v \end{array}\right)\Big|_{t=0}= \left( \begin{array}{cc} u^0\\v^0\end{array}\right)$$
and
$$\Gamma\left( \begin{array}{cc} u\\v \end{array}\right)\Big|_{t=T}= \left( \begin{array}{cc} u^1\\v^1\end{array}\right)+\left(\begin{array}{cc} v \\ \nu (T,u,v)\end{array}\right)-\left(\begin{array}{cc} v \\ \nu (T,u,v)\end{array}\right)= \left( \begin{array}{cc} u^1\\v^1\end{array}\right).$$
If we show that the map $\Gamma$ is a contraction in an appropriate metric space, then its fixed point $(u,v)$ is the solution of \eqref{gg1_new}-\eqref{gg2_new} with $\vec{h}_1$ and $\vec{g}_1$ defined by \eqref{controli}, satisfying $u(\cdot,T)=u^1(\cdot)$ and $v(\cdot,T)=v^1(\cdot)$. In order to prove the existence of the fixed point we apply the Banach fixed point theorem to the restriction of $\Gamma$ on closed ball
$$B_r=\left\{ ( u, v)  \in  \ZZ_T : \left\| ( u, v )\right\|_{\ZZ_T} \le r \right\},$$
for some $r>0$.

\begin{itemize}
\item[(a)]\textit{$\Gamma$ maps $B_r$ in itself.}
\end{itemize}
Indeed, as in the proof of Theorem \ref{nonlinearteo}, we obtain that there exists a constant $C_1>0$ such that

\begin{equation*}
\left\|\Gamma\left( \begin{array}{cc} u\\v \end{array}\right) \right\|_{\ZZ_T}  \le   C_1\delta +C_2(r+1)r,
\end{equation*}
where $C_2$ is a constant depending only $T$. Thus, if we select $r$ and $\delta$ satisfying $$ r=2C_1\delta$$ and $$ 2C_1C_2\delta +C_2\leq \frac12,$$ the operator $\Gamma$ maps $B_r$ into itself for any $(u,v) \in \ZZ_T$.
\begin{itemize}
\item[(b)]\textit{$\Gamma$ is contractive.}
\end{itemize}
In fact, proceeding as the proof of Theorem \ref{nonlinearteo}, we obtain
\begin{equation*}
\left\|\Gamma\left( \begin{array}{cc} u\\v \end{array}\right) - \Gamma\left( \begin{array}{cc} \widetilde{u}\\\widetilde{v} \end{array}\right) \right\|_{\ZZ_T}  \le   C_3(r+1)r \left\|\left( \begin{array}{cc} u-\widetilde{u}\\v-\widetilde{v} \end{array}\right) \right\|_{\ZZ_T},
\end{equation*}
for any $(u,v), (\widetilde{u},\widetilde{v}) \in B_r$ and $C_3$ constant depending only $T$. Thus, choosing $\delta >0$, such that $$\gamma= 2C_2C_3\delta +C_3 < 1,$$ we obtain
\begin{equation*}
\left\|\Gamma\left( \begin{array}{cc} u\\v \end{array}\right) - \Gamma\left( \begin{array}{cc} \widetilde{u}\\\widetilde{v} \end{array}\right) \right\|_{\ZZ_T}  \le   \gamma \left\|\left( \begin{array}{cc} u-\widetilde{u}\\v-\widetilde{v} \end{array}\right) \right\|_{\ZZ_T}.
\end{equation*}
Therefore, the map $\Gamma$ is a contraction.

Thus, from (a) and (b), $\Gamma$ has a fixed point in $B_r$ by the Banach fixed point Theorem and its fixed point is the desired solution. The proof of Theorem \ref{main} is archived.
\end{proof}

\section{Further Comments}

The following remarks are now in order:

\vglue 0.2 cm

\noindent$\bullet$ In \cite{micuortegapazoto2009}, it was proved that the system \eqref{gg1} with the boundary conditions
\begin{equation}\label{mop}
\left\lbrace\begin{tabular}{l l l l}
$u(0,t) = 0$,     & $u(L,t) = h_1(t)$, & $u_{x}(L,t) = h_2(t)$ & in $(0,T)$,\\
$v(0,t) = 0$,     & $v(L,t) = g_1(t)$, & $v_{x}(L,t) = g_2(t)$ & in $(0,T)$,
\end{tabular}\right.
\end{equation}
is exactly controllable in $L^2(0,L)$ when $h_1, g_1\in H^1_0(0, T)$ and $h_2, g_2 \in L^2(0,T)$ (see Theorem A). By using the tools developed in this paper, more precisely, Lemma \ref{sharp_int}, an improvement of the regularity of the control can be obtained. In this case, the control $(h_1, g_1, h_2,g_2)$ can be found in the space  $H^{\frac{1}{3}}(0,T)\times  H^{\frac{1}{3}}(0,T)\times L^2(0,T)\times L^2(0,T)$.

\vglue 0.2 cm

\noindent$\bullet$ Another case that can be treated is the following one
\begin{equation}\label{ggln4b'}
\left\lbrace\begin{tabular}{l l l l}
$u(0,t) = h_0(t)$,     & $u(L,t) = h_1(t)$,      & $u_{x}(L,t) = h_2(t)$ & in $(0,T)$,\\
$v(0,t) =0$,           & $v(L,t) = 0$,           & $v_{x}(L,t) = g_2(t)$ & in $(0,T)$.
\end{tabular}\right.
\end{equation}
By using the same ideas of the proof of Theorem \ref{main}, we can prove that system \eqref{ggln4}-\eqref{ggln4b'} is exactly controllable for any time $T>0$ if $L \in (0,\infty) \setminus \FF'_r$.

\vglue 0.2 cm

\noindent$\bullet$ Concerning the exact boundary controllability of the system \eqref{gg1} with one control, our approach can be applied to the following configuration:
\begin{equation}\label{gglnb3'}
\left\lbrace\begin{tabular}{l l l l}
$u(0,t) = 0$ & $u(L,t) = 0$ & $u_{x}(L,t) = 0$,& in $(0,T)$,\\
$v(0,t) =0,$ & $v(L,t) = 0,$ & $v_{x}(L,t) =g_2(t)$,& in $(0,T)$.
\end{tabular}\right.
\end{equation}
The proof of this case is analogous to (ii) of Theorem \ref{main}.

\subsection*{Acknowledgments}

Fernando A. Gallego was supported by CAPES (Brazil) and Ademir F. Pazoto was partially supported by CNPq (Brazil).

\end{document}